\documentclass[reqno,12pt]{amsart}
\usepackage{graphics}
\usepackage{amssymb, amsmath}
\usepackage{version}
\usepackage{fancybox}
\usepackage{bm}
\usepackage{xcolor}
\usepackage{mathrsfs}
\usepackage{mathtools} 
\theoremstyle{theorem}
\newtheorem{theorem}{\sc Theorem}[section]
\newtheorem{lemma}[theorem]{\sc Lemma}

\newtheorem{proposition}[theorem]{\sc Proposition}
\newtheorem{corollary}[theorem]{\sc Corollary}

\theoremstyle{definition}
\newtheorem{definition}[theorem]{\sc Definition}

\theoremstyle{remark}
\newtheorem{remark}[theorem]{\sc Remark}

\makeatletter

\@addtoreset{equation}{section}

\renewcommand{\d}{\text{\rm d}}

\newcommand{\vep}{\varepsilon}

\newcommand{\e}{{\rm e}}

\newcommand{\R}{\mathbb{R}}
\newcommand{\N}{\mathbb{N}}
\newcommand{\Linv}{(-\Delta)^{-1}}

\newcommand{\Lphi}{\mathcal L_\phi}
\newcommand{\Ls}{\mathcal L_s}
\renewcommand{\H}{\mathcal{H}}

\allowdisplaybreaks[4]

\newcommand{\prf}[1]{}  
\begin{document}
\title[Optimal rate of convergence for fast diffusion]{Optimal rate of convergence to nondegenerate asymptotic profiles for fast diffusion in domains\prf{\\{\tt [Extended version]}}}
%
%
\author{Goro Akagi}
\address[Goro Akagi]{Mathematical Institute and Graduate School of Science, Tohoku University, Aoba, Sendai 980-8578, Japan}
\email{goro.akagi@tohoku.ac.jp}
\author{Yasunori Maekawa}
\address[Yasunori Maekawa]{Department of Mathematics, Kyoto University, Kitashirakawa Oiwakecho, Sakyo-ku, Kyoto 606-8502, Japan}
\email{maekawa.yasunori.3n@kyoto-u.ac.jp}
\thanks{G.A.~is supported by JSPS KAKENHI Grant Numbers JP21KK0044, JP21K18581, JP20H01812 and JP20H00117. Y.M.~is supported by JSPS KAKENHI Grant Numbers 20H00118, 19H05597, 21H04433, 21H00991, 20K03698, and 23H01082. This work was also supported by the Research Institute for Mathematical Sciences, an International Joint Usage/Research Center located in Kyoto University.}
\date{\today}
\maketitle
\begin{abstract}
This paper is concerned with the Cauchy-Dirichlet problem for fast diffusion equations posed in bounded domains, where every energy solution vanishes in finite time and a suitably rescaled solution converges to an asymptotic profile. Bonforte and Figalli (CPAM, 2021) first proved an exponential convergence to nondegenerate \emph{positive} asymptotic profiles for nonnegative rescaled solutions in a weighted $L^2$ norm for smooth bounded domains by developing a \emph{nonlinear entropy method}. However, the optimality of the rate remains open to question. In the present paper, their result is fully extended to possibly \emph{sign-changing} asymptotic profiles as well as \emph{general} bounded domains by improving an \emph{energy method} along with a \emph{quantitative gradient inequality} developed by the first author (ARMA, 2023). Moreover, a (quantitative) exponential stability result for \emph{least-energy} asymptotic profiles follows as a corollary, and it is further employed to prove the optimality of the exponential rate.
\end{abstract}

\section{Introduction}\label{S:Intro}

Let $\Omega$ be any bounded domain of $\mathbb R^N$ with boundary $\partial \Omega$. There are a great number of contributions to the study of \emph{nonlinear diffusion equations} posed on bounded domains, that is,
\begin{equation}\label{NDE}
\partial_t \rho = \Delta \rho^m \ \mbox{ in } \Omega \times (0,\infty),
\end{equation}
where $\partial_t = \partial/\partial t$, $\rho = \rho(x,t)$ denotes the density of a diffusing substance and the diffusion coefficient $D$ scales with $\rho^{m-1}$ for an exponent $0 < m < \infty$. In particular, the case $0 < m < 1$ (respectively, $m > 1$) is called a \emph{fast diffusion equation} (respectively, \emph{porous medium equation}) and classified as a \emph{singular diffusion} (respectively, \emph{degenerate diffusion}). 

In the present paper, we deal with (possibly sign-changing) solutions to the Cauchy-Dirichlet problem for the fast diffusion equation,
\begin{alignat}{4}
\partial_t \left( |u|^{q-2}u \right) &= \Delta u
 \quad && \mbox{ in } \Omega \times (0, \infty),\label{eq:1.1}\\
 u &= 0 && \mbox{ on } \partial \Omega \times (0, \infty), \label{eq:1.2}\\
 u &= u_0 && \mbox{ on } \Omega \times \{0\}.\label{eq:1.3}
\end{alignat}
Of course, \eqref{eq:1.1} is transformed from \eqref{NDE} by setting $u = \rho^m$ and $q-1 = 1/m$, and vice versa. Throughout this paper, we assume that
\begin{equation}\label{hypo}
u_0 \in H^1_0(\Omega) \setminus \{0\}, \quad 2 < q < 2^* := \dfrac{2N}{(N-2)_+}.
\end{equation}
This problem was studied by Berryman and Holland in~\cite{BH78,BH80}, which were motivated in order to give a theoretical interpretation to the experimental observation of anomalous diffusion of hydrogen plasma across a purely poloidal octupole magnetic field that \emph{after a few milliseconds the density profile always evolves into a fixed shape {\rm (}the ``normal mode''{\rm )} which then decays in time} based on the Okuda-Dawson model $D \sim \rho^{-1/2}$ (i.e., the case $q = 3$) proposed in~\cite{Okuda-Dawson}.

Let us recall \emph{qualitative} results on asymptotic behavior of (weak) solutions to \eqref{eq:1.1}--\eqref{eq:1.3}. Due to the homogeneous Dirichlet boundary condition, the diffusion coefficient $D$ diverges on the boundary (see \eqref{NDE}). As a result, every weak solution $u = u(x,t)$ of \eqref{eq:1.1}--\eqref{eq:1.3} vanishes at a finite time $t_*$, which is uniquely determined by the initial datum $u_0$ (see~\cite{Sabinina62,BC,Diaz88,HerreroVazquez88}); hence, we denote $t_* = t_*(u_0)$. Moreover, Berryman and Holland~\cite{BH80} proved that the extinction rate of the \emph{positive classical} solution $u(\cdot,t)$ is just $(t_*-t)_+^{1/(q-2)}$ as $t \nearrow t_*$, that is,
\begin{equation}\label{ext_time}
c_1 (t_* - t)_+^{1/(q-2)} \leq \|u(\cdot,t)\|_{H^1_0(\Omega)} 
\leq c_2 (t_* - t)_+^{1/(q-2)}
\end{equation}
with $c_1, c_2 > 0$ for all $t \geq 0$, provided that $u_0 \not\equiv 0$; furthermore, this fact is extended to (possibly) \emph{sign-changing weak solutions} by~\cite{Kwong88-3,DKV91,SavareVespri,AK13} (see also~\cite{BGV08,BV10,JiXi23,JiXi22+}). Therefore the \emph{asymptotic profile} $\phi(x)$ of $u(x,t)$ is defined by
\begin{equation}\label{ap}
\phi(x) = \lim_{t\nearrow t_*} (t_*-t)^{-1/(q-2)} u(x,t) \not\equiv 0 \ \mbox{ in } H^1_0(\Omega),
\end{equation}
which corresponds to the \emph{fixed shape of the density profile} concerned in~\cite{BH78,BH80}.
Apply the change of variables, 
\begin{equation}\label{cv}
v(x,s) = (t_* - t)^{-1/(q-2)} u(x,t)
\ \mbox{ and } \  s = \log (t_*/(t_* - t))
\end{equation}
for $t \in [0,t_*)$. Then the asymptotic profile $\phi(x)$ is reformulated as the limit of $v(x,s)$ as $s \to \infty$. Moreover, $v=v(x,s)$ turns out to be an energy solution of the following Cauchy-Dirichlet problem:
\begin{alignat}{4}
 \partial_s \left( |v|^{q-2}v \right) &= \Delta v + \lambda_q |v|^{q-2}v
\quad && \mbox{ in } \Omega \times (0, \infty),\label{eq:1.6}\\
 v &= 0 && \mbox{ on } \partial \Omega \times (0, \infty), \label{eq:1.7}\\
 v &= v_0 && \mbox{ on } \Omega \times \{0\}\label{eq:1.8}
\end{alignat}
with $\lambda_q := (q-1)/(q-2) > 0$ and $v_0 := t_*(u_0)^{-1/(q-2)}u_0$. Here we note that \eqref{eq:1.6} along with \eqref{eq:1.7} can also be formulated as a (generalized) gradient flow of the form,
$$
\partial_s \left( |v|^{q-2}v \right)(s) = - J'(v(s)) \ \mbox{ in } H^{-1}(\Omega), \quad s > 0,
$$
where $J' : H^1_0(\Omega) \to H^{-1}(\Omega)$ denotes the Fr\'echet derivative of the energy functional,
$$
J(w) := \frac 12 \int_\Omega |\nabla w(x)|^2 \, \d x - \frac{\lambda_q}q \int_\Omega |w(x)|^q \, \d x \quad \mbox{ for } \ w \in H^1_0(\Omega).
$$
Here and henceforth, we may denote $v(s) = v(\cdot,s)$ for $s \geq 0$. Moreover, it is also noteworthy that $v_0$ lies on the set,
\begin{align}
\mathcal{X} :=& \ \{ t_*(u_0)^{-1/(q-2)}u_0 \colon u_0 \in H^1_0(\Omega) \setminus \{0\}\} \label{phase_set}\\
=& \ \{w \in H^1_0(\Omega) \colon t_*(w) = 1\},\nonumber
\end{align}
which is an invariant set of the dynamical system generated by \eqref{eq:1.6}--\eqref{eq:1.8} and plays a role of the phase set in stability analysis of asymptotic profiles (see Definition \ref{D:stbl} below and~\cite{AK13} for more details). Moreover, by virtue of \eqref{ext_time}, we see that
\begin{equation}\label{v-bdd}
0 < c_1 \leq \|v(s)\|_{H^1_0(\Omega)} \leq c_2 < +\infty \quad \mbox{ for } \ s \geq 0.
\end{equation}
Hence the norm $\|v(\cdot,s)\|_{H^1_0(\Omega)}$ can neither vanish nor grow up to infinity (cf.~see~\cite[Proposition 10]{AK13}).

Berryman and Holland~\cite{BH80} proved that any \emph{positive classical} solution $v(\cdot,s_n)$ of \eqref{eq:1.6}--\eqref{eq:1.8} converges strongly in $H^1_0(\Omega)$ to a nontrivial solution $\phi = \phi(x)$ to the Dirichlet problem,
\begin{alignat}{4}
 - \Delta \phi &= \lambda_q |\phi|^{q-2}\phi \quad && \mbox{ in } \Omega,
\label{eq:1.10}\\
\phi &= 0 && \mbox{ on } \partial \Omega,
\label{eq:1.11}
\end{alignat}
for some sequence $s_n \to +\infty$ and, in particular, if $N = 1$, then $v(s) \to \phi$ as $s \to +\infty$. Such a \emph{quasi-convergence} result was extended to (possibly) \emph{sign-changing weak solutions} in~\cite{Kwong88-3,DKV91,SavareVespri,BGV08,BV10,AK13}. More precisely, the following theorem holds true:
\begin{theorem}[\cite{BH80,Kwong88-3,DKV91,SavareVespri,AK13}]
Under the assumption \eqref{hypo}, let $u$ be a {\rm (}possibly sign-changing{\rm )} energy solution of {\rm \eqref{eq:1.1}--\eqref{eq:1.3}} and let $t_* \in (0,\infty)$ be the extinction time of $u$. Then for any increasing sequence $t_n \to t_*$, there exist a subsequence $(n')$ of $(n)$ and a nontrivial solution $\phi \in H^1_0(\Omega)\setminus \{0\}$ of {\rm \eqref{eq:1.10}, \eqref{eq:1.11}} such that
 \begin{equation}\label{qconv}
  \lim_{t_{n'} \to t_*}
   \|(t_* - t_{n'})^{-1/(q-2)} u(t_{n'}) - \phi \|_{H^1_0(\Omega)} = 0,
 \end{equation}
equivalently,
 \begin{equation*}
  \lim_{s_{n'} \to \infty}
   \|v(s_{n'}) - \phi \|_{H^1_0(\Omega)} = 0,
 \end{equation*}
where $v$ and $s_n$ are defined as in \eqref{cv} for $u$ and $t_n$, respectively.
\end{theorem}
Moreover, Feireisl and Simondon~\cite{FeiSim00} proved convergence of any \emph{nonnegative} weak solution $v = v(x,s) \geq 0$ for \eqref{eq:1.6}--\eqref{eq:1.8} to a positive solution $\phi$ for \eqref{eq:1.10}, \eqref{eq:1.11} in $C(\overline{\Omega})$ as $s \to +\infty$ by developing a \L ojasiewicz-Simon gradient inequality. Furthermore, based on this along with the so-called \emph{Global Harnack Principle} (GHP for short), which is valid for bounded $C^2$ domains and developed in~\cite{BGV08}, that is, for any $\delta > 0$, there exist constants $c_3,c_4 > 0$ such that
\begin{equation}\label{GHP}
c_3 \leq \frac{v(\cdot,s)}{\mathrm{dist}(\cdot,\Omega)} \leq c_4 \ \mbox{ on } \overline\Omega \ \mbox{ for } \ s > \delta,
\end{equation}
where $\mathrm{dist}(x,\partial \Omega) := \inf_{y \in \partial \Omega} |x-y| \asymp \phi(x) > 0$, Bonforte, Grillo and Vazquez~\cite{BGV12} proved convergence of the \emph{relative error},
\begin{equation}\label{rel-err-conv}
h(s) := (v(s)-\phi)/\phi \ \mbox{ in } C(\overline{\Omega}) \ \mbox{ as } \ s \to +\infty
\end{equation}
for positive solutions (see also~\cite[Theorem 4.1]{BF21} for a quantitative result, which also gives an alternative proof to the above).

As for \emph{quantitative} results, developing a \emph{nonlinear entropy method}, Bonforte and Figalli~\cite{BF21} proved a sharp rate of convergence for nonnegative $v = v(x,s)$ in the \emph{relative entropy}, 
\begin{equation}\label{ent-conv}
\mathsf{E}(s) := \int_\Omega |v(x,s)-\phi(x)|^2 \phi(x)^{q-2} \, \d x \leq C \e^{-\lambda_0 s} \ \mbox{ for } \ s > 0,
\end{equation}
where $\lambda_0 := 2\nu_k/(q-1)$ and $\nu_k$ is the least \emph{positive} eigenvalue of the weighted eigenvalue problem
\begin{equation}\label{Lphi-ep}
\Lphi e = \nu |\phi|^{q-2} e \ \mbox{ in } \Omega, \quad e = 0 \ \mbox{ on } \partial \Omega
\end{equation}
for the linearized operator $\Lphi := - \Delta - \lambda_q (q-1) |\phi|^{q-2}$, provided that $\phi$ is \emph{positive} and \emph{nondegenerate} (i.e., $\Lphi$ has no zero eigenvalue) and $\partial\Omega$ is smooth (at least of class $C^2$). The above rate of convergence seems \emph{sharp} in view of a formal linearization (see~\cite[\S 2]{BF21}). 

Furthermore, an alternative approach based on an energy method along with a quantitative gradient inequality is developed in~\cite{A21} to (directly) prove that
\begin{equation}\label{h10-conv}
\|v(s) - \phi\|_{H^1_0(\Omega)}^2 \leq C \e^{-\lambda_0 s} \ \mbox{ for } \ s > 0, 
\end{equation}
which also immediately yields \eqref{ent-conv}, for nonnegative solutions $v = v(x,s)$ to \eqref{eq:1.6}--\eqref{eq:1.8} and positive nondegenerate solutions $\phi = \phi(x)$ to \eqref{eq:1.10}, \eqref{eq:1.11} for any bounded $C^{1,1}$ domains. Furthermore, in~\cite{A21}, it is also proved for (possibly) sign-changing solutions that \eqref{h10-conv} is satisfied with $\lambda_0$ replaced by any
$$
0 < \lambda < \frac 2{q-1} C_q^{-2} \|\phi\|_{L^q(\Omega)}^{-(q-2)} \frac{\nu_k}{\nu_k + \lambda_q(q-1)},
$$
where $C_q$ stands for the best constant of a Sobolev-Poincar\'e inequality and which cannot however reach the sharp exponent $\lambda_0$ even for least-energy solutions to \eqref{eq:1.10}, \eqref{eq:1.11} (see Remark 3.2 of~\cite{A21}).

On the other hand, the topology of the convergence can be improved with the aid of optimal boundary regularity results developed by Jin and Xiong in~\cite{JiXi23,JiXi22+}, which is also motivated from a long-standing open question posed in~\cite{BH80}. More precisely, Jin and Xiong~\cite{JiXi23,JiXi22+} proved the optimal boundary regularity (e.g., $\partial_t^\ell u(\cdot,t) \in C^{q+1}(\overline\Omega)$ for any $\ell \in \N$) of nonnegative solutions to \eqref{eq:1.1}--\eqref{eq:1.3}, which is consistent with the regularity of separable solutions $u = u(x,t)$ to \eqref{eq:1.1}, \eqref{eq:1.2}, in smooth bounded domains, by developing Schauder estimates for some linear parabolic equations with degenerate coefficients asymptotic to $\mathrm{dist}(x,\partial \Omega)^{q-2}$ (in front) of the time-derivative, with the aid of the GHP \eqref{GHP}. Moreover, based on the optimal boundary regularity result, they also proved that \eqref{ent-conv} can be improved up to
$$
\left\|\frac{v(s)}\phi - 1\right\|_{C^q(\overline\Omega)} \leq C \e^{-\lambda_0 s} \ \mbox{ for } \ s > 1
$$
for nonnegative solutions in \emph{smooth} bounded domains.

Furthermore, Choi, McCann and Seis~\cite{McCann23} proved a dichotomy result on the rate of convergence of $v = v(x,s) \geq 0$ to (possibly) \emph{degenerate} positive solutions $\phi = \phi(x)$; more precisely, either of $\mathsf{E}(s) \lesssim \e^{-\lambda_0 s}$ or $\mathsf{E}(s) \gtrsim s^{-1}$ always holds (cf.~see also~\cite{JiXi20+}). They observed that the relative error $h(\cdot,s) := (v(\cdot,s)-\phi)/\phi$ solves
$$
\partial_s h + L_\phi(h) = \mathcal{N}(h),
$$
where $L_\phi$ is a linear elliptic operator including coefficients associated with $\phi$ and $\mathcal{N}$ is a nonlinear perturbation, which still involves $\partial_s h$ but can be handled as a small perturbation for $h$ small enough, by proving a smoothing estimate for $\partial_s h$. Then the dichotomy result follows from an ODE analysis of a reduced system. This dichotomy result also enables us to derive the sharp rate of convergence \eqref{ent-conv} for nondegenerate positive asymptotic profiles in smooth bounded domains (see also~\cite{CS23+}). We further refer the reader to the recent article~\cite{BF23+} for a comprehensive survey on this field.

As seen from the above, convergence to \emph{positive} asymptotic profiles in \emph{smooth} bounded domains has been well studied; on the other hand, results for \emph{sign-changing} asymptotic profiles are still limited. In particular, the sharp rate of convergence as in \eqref{ent-conv} and \eqref{h10-conv} has not yet been proved for sign-changing solutions. Actually, the nonlinear entropy method is deeply based on the GHP, and hence, the positivity of the asymptotic profile may be indispensable. The energy method developed in~\cite{A21} is applicable to sign-changing asymptotic profiles; however, the conclusion for sign-changing asymptotic profiles does not reach the sharp rate of convergence.

Another open question in this field is the \emph{optimality of the convergence rate} (see \eqref{ent-conv}) even for positive asymptotic profiles; indeed, there seems to be no proof, although it may be expected to be optimal in view of a formal linearlized analysis (see~\cite[\S 2]{BF21}). On the other hand, as for the porous medium case (i.e., $m > 2$ and $1 < q < 2$), the \emph{optimal} rate of convergence to the (unique) positive asymptotic profile was determined by means of the classical comparison argument in~\cite{ArPel81}, and moreover, a finer asymptotics has also been investigated in a recent paper~\cite{JiROXi22+}.

The first purpose of the present paper is to prove \eqref{h10-conv} for each (possibly) sign-changing solution $v = v(x,s)$ of \eqref{eq:1.6}--\eqref{eq:1.8} which converges to a nondegenerate asymptotic profile as $s \to +\infty$. We stress that our method of proof is completely free from both the relative error convergence \eqref{rel-err-conv} and smoothing estimates, which have been developed for nonnegative solutions, but only a few results are known for sign-changing ones. Furthermore, compared to the previous results on nonnegative solutions based on the relative error convergence and smoothing estimates in~\cite{BF21,JiXi23,JiXi22+,McCann23} as well as results in~\cite{A21}, we need no assumption on the smoothness of domains. The second purpose of the present paper is to prove the \emph{optimality} of the convergence rate (see \eqref{ent-conv}) to nondegenerate \emph{least-energy asymptotic profiles} (see below for definition) with the aid of the improved convergence result mentioned above.

The main results of the present paper are stated as follows:
\begin{theorem}[Sharp rate of convergence]\label{T:sc-conv}
Let $\Omega$ be any bounded domain of $\R^N$ with boundary $\partial \Omega$. Under the assumption \eqref{hypo}, let $v = v(x,s)$ be a {\rm (}possibly{\rm )} \emph{sign-changing} energy solution to \eqref{eq:1.6}--\eqref{eq:1.8} and let $\phi = \phi(x)$ be a nondegenerate nontrivial solution to \eqref{eq:1.10}, \eqref{eq:1.11} such that $v(s_n) \to \phi$ strongly in $H^1_0(\Omega)$ for some sequence $s_n \to +\infty$. Then there exists a constant $C \geq 0$ such that
\begin{align}
0 \leq J(v(s)) - J(\phi) &\leq C \e^{-\lambda_0 s} \quad \mbox{ for all } \ s \geq 0,\label{J-conv}\\
\|v(s) - \phi\|_{H^1_0(\Omega)}^2 &\leq C \e^{-\lambda_0 s} \quad \mbox{ for all } \ s \geq 0,\label{H10-ec}
\end{align}
where $\lambda_0$ is defined as in \eqref{ent-conv}.
\end{theorem}

In~\cite{A21}, some examples of nondegenerate sign-changing asymptotic profiles are exhibited. In particular, for $N \geq 2$, nondegenerate sign-changing asymptotic profiles are constructed in dumbbell domains, and moreover, their exponential stability is also proved under certain symmetry of initial data; actually, sign-changing asymptotic profiles are never asymptotically stable for general initial data (see~\cite[Theorem 3]{AK13}). Furthermore, the exponential convergence \eqref{H10-ec} can also be rephrased with the original variables as follows (cf.~see \eqref{ext_time}, \eqref{qconv}):
$$
\| (t_*-t)^{-1/(q-2)} u(t) - \phi \|_{H^1_0(\Omega)}^2 \leq C \left(\frac{t_* - t}{t_*}\right)_+^{\lambda_0} \quad \mbox{ for } \ t \geq 0.
$$

In what follows, the \emph{least-energy solutions} to \eqref{eq:1.10}, \eqref{eq:1.11} (or \emph{least-energy asymptotic profiles}) mean nontrivial weak solutions to \eqref{eq:1.10}, \eqref{eq:1.11} minimizing the energy $J$ among all the weak nontrivial solutions to \eqref{eq:1.10}, \eqref{eq:1.11}. The least positive eigenvalue of \eqref{Lphi-ep} for any nondegenerate least-energy asymptotic profile $\phi$ is the second one, that is, $k = 2$ by~\cite{Lin}. Now we have the following corollary, which improves an exponential stability result in~\cite[Corollary 1.3]{A21}:
\begin{corollary}[Quantitative exponential stability of least-energy profiles]\label{C:stbl}
Let $\Omega$ be any bounded domain of $\R^N$ with boundary $\partial \Omega$. Assume \eqref{hypo} and let $\phi$ be a nondegenerate \emph{least-energy} solution to \eqref{eq:1.10}, \eqref{eq:1.11}. Then there exists constants $\delta_0, C > 0$ satisfying the following\/{\rm :} Let $v_0 \in \mathcal{X}$ be such that $\|v_0 - \phi\|_{H^1_0(\Omega)} < \delta_0$ and let $v = v(x,s)$ be the energy solution to \eqref{eq:1.6}--\eqref{eq:1.8} such that $v(0) = v_0$. Then it holds that
\begin{align}
0 \leq J(v(s)) - J(\phi) &\leq C \left(J(v_0)-J(\phi)\right) \e^{-\lambda_0 s},\label{J-stbl}\\
\|v(s) - \phi\|_{H^1_0(\Omega)}^2 &\leq C \left(J(v_0)-J(\phi)\right) \e^{-\lambda_0 s},\label{H10-stbl}
\end{align}
where $\lambda_0$ is defined as in \eqref{ent-conv}, for all $s \geq 0$.
\end{corollary}

The corollary mentioned above further enables us to prove the optimality of the rate of convergence provided in Theorem \ref{T:sc-conv} (and Corollary \ref{C:stbl}) for \emph{least-energy} asymptotic profiles.
\begin{theorem}[Optimality of the convergence rate]\label{T:opt}
Let $\Omega$ be any bounded domain of $\R^N$ with boundary $\partial \Omega$. Assume \eqref{hypo} and let $\phi$ be a nondegenerate \emph{least-energy} solution to \eqref{eq:1.10}, \eqref{eq:1.11} and let $\mathbb{P}_2$ be the spectral projection onto the eigenspace $E_2$ corresponding to the least positive eigenvalue $\nu_2$ of the eigenvalue problem \eqref{Lphi-ep}. Let $\xi_\vep \in H^1_0(\Omega)$, $\vep > 0$ be such that
\begin{equation}\label{opt-hyp}
\begin{cases}
\|\xi_\vep\|_{H^1_0(\Omega)} = O(\vep) \ \mbox{ as } \ \vep \to 0_+,\\
\liminf_{\vep \to 0_+} \vep^{-1} \|\mathbb{P}_2(\xi_\vep)\|_{H^1_0(\Omega)} > 0.
\end{cases}
\end{equation}
Set $u_{0,\vep} := \phi + \xi_\vep$ and $v_{0,\vep} := t_*(u_{0,\vep})^{-1/(q-2)} u_{0,\vep} \in \mathcal X$. Let $v_\vep = v_\vep(x,s)$ be the energy solution to \eqref{eq:1.6}--\eqref{eq:1.8} for the initial datum $v_{0,\vep}$. Then there exists $\vep_0 > 0$ such that, for any $\vep \in (0,\vep_0)$,
\begin{align}
c_{\vep} \e^{-\lambda_0 s} &\leq \int_\Omega |v_\vep(s)-\phi|^2 \phi^{q-2} \, \d x\nonumber\\
&\leq C \|v_\vep(s)-\phi\|_{H^1_0(\Omega)}^2
\leq C_{\vep} \e^{-\lambda_0 s} \quad \mbox{ for } \ s \geq 0
\label{opt-est}
\end{align}
for some positive constants $c_{\vep}, C_{\vep}, C > 0$. Hence the rate of convergence provided in Theorem {\rm \ref{T:sc-conv} (}and Corollary {\rm \ref{C:stbl})} is optimal for least-energy asymptotic profiles.
\end{theorem}

Moreover, in \S \ref{S:faster}, we shall also construct an initial datum $v_0 \in \mathcal X$ for which the energy solution $v = v(x,s)$ to \eqref{eq:1.6}--\eqref{eq:1.8} converges to $\phi$ faster than $\e^{- \frac{\nu_2}{q-1} s}$ as $s \to +\infty$ (see Theorem \ref{T:faster} below).

\prf{
Moreover, we have the following corollary, which completely extends the main result of~\cite{BF21} to (possibly) sign-changing solutions as well as to less regular domains:
\begin{corollary}\label{C:ent-conv}
Under the same assumptions as in Theorem {\rm \ref{T:sc-conv}}, there exists a constant $C \geq 0$ such that
$$
\int_\Omega |v(s)-\phi|^2 |\phi|^{q-2} \, \d x \leq C \e^{-\lambda_0 s} \ \mbox{ for all } \ s \geq 0,
$$
where $\lambda_0$ is defined as in \eqref{ent-conv}.
\end{corollary}

\begin{proof}
Since $\phi$ is bounded in $\Omega$, the conclusion immediately follows from \eqref{H10-ec} as well as the observation,
$$
\int_\Omega |w|^2 |\phi|^{q-2} \, \d x \leq \|\phi\|_{L^\infty(\Omega)}^{q-2} \|w\|_{L^2(\Omega)}^2 \quad \mbox{ for } \ w \in L^2(\Omega),
$$
which completes the proof.
\end{proof}
}

The present paper is composed of six sections. In Section \ref{S:pre}, we recall some preliminary facts, e.g., regularity of energy solutions and notions of stability for asymptotic profiles. Section \ref{S:sc-conv} is devoted to proofs of Theorem \ref{T:sc-conv} and Corollary \ref{C:stbl}. In Section \ref{S:alt}, we also discuss an alternative proof of Theorem \ref{T:sc-conv} as an independent interest. In Section \ref{S:opt}, the optimality of the convergence rate to least-energy asymptotic profiles is proved (see Theorem \ref{T:opt}). Finally, Section \ref{S:faster} presents a construction of well-prepared initial data for which rescaled solutions converge to least-energy asymptotic profiles faster than the optimal convergence rate (see Theorem \ref{T:faster} below).

\bigskip
\noindent
{\bf Notation.} Let $A \subset \R^N$ be an $N$-dimensional Lebesgue measurable set and denote by $\mathcal{M}(A)$ the set of all Lebesgue measurable functions defined on $A$ with values in $\R$. We denote by $C$ a generic nonnegative constant which may vary from line to line. We denote by $H^{-1}(\Omega)$ the dual space of the Sobolev space $H^1_0(\Omega)$ equipped with the inner product $(u,v)_{H^1_0(\Omega)} = \int_\Omega \nabla u \cdot \nabla v \, \d x$ for $u,v \in H^1_0(\Omega)$. Moreover, $\langle \cdot, \cdot \rangle_{H^1_0(\Omega)}$ stands for the duality pairing between $H^1_0(\Omega)$ and $H^{-1}(\Omega)$. Furthermore, an inner product of $H^{-1}(\Omega)$ is naturally defined as
\begin{equation}\label{H-1product}
(f,g)_{H^{-1}(\Omega)} = \langle f, (-\Delta)^{-1}g\rangle_{H^1_0(\Omega)} \quad \mbox{ for } \ f,g \in H^{-1}(\Omega),
\end{equation}
which also gives $\|f\|_{H^{-1}(\Omega)}^2 = (f,f)_{H^{-1}(\Omega)}$ for $f \in H^{-1}(\Omega)$. Then $-\Delta$ is a duality mapping (Riesz mapping) between $H^1_0(\Omega)$ and $H^{-1}(\Omega)$, that is,
\begin{align*}
\|u\|_{H^1_0(\Omega)}^2 &= \|-\Delta u\|_{H^{-1}(\Omega)}^2 = \langle -\Delta u, u \rangle_{H^1_0(\Omega)},\\
\|f\|_{H^{-1}(\Omega)}^2 &= \|(-\Delta)^{-1} f\|_{H^1_0(\Omega)}^2 = \langle f, (-\Delta)^{-1} f \rangle_{H^1_0(\Omega)}
\end{align*}
for $u \in H^1_0(\Omega)$ and $f \in H^{-1}(\Omega)$.

\section{Preliminaries}\label{S:pre}

In this section, we shall collect preliminary material for later use. Throughout this paper, we are concerned with \emph{energy solutions} defined by
\begin{definition}[Energy solution]\label{D:sol}
 A function $u : \Omega \times (0,\infty) \to \mathbb R$ is called an \emph{{\rm (}energy{\rm )} solution} of \eqref{eq:1.1}--\eqref{eq:1.3}, if the following
 conditions hold true\/{\rm :}
\begin{itemize}
 \item $u \in L^\infty(0,\infty;H^1_0(\Omega))$ and 
       $|u|^{q-2}u \in W^{1,\infty}(0,\infty;H^{-1}(\Omega))$,
 \item for a.e.~$t \in (0, \infty)$, it holds that
       \begin{align}
  \left\langle \partial_t \left(|u|^{q-2}u\right)(t), \phi \right\rangle_{H^1_0}
   + \int_\Omega \nabla u(x,t) \cdot \nabla \phi(x) \,\d x
   = 0 \label{EQ}\\
	\nonumber
	 \qquad \mbox{ for all } \ \phi \in H^1_0(\Omega),
       \end{align}
       where $\langle \cdot, \cdot \rangle_{H^1_0}$ denotes the duality
       pairing between $H^1_0(\Omega)$ and its dual space
       $H^{-1}(\Omega)$,
 \item $u(\cdot,t) \to u_0$ strongly in $H^1_0(\Omega)$ as $t \to 0_+$.
\end{itemize}
Moreover, energy solutions of \eqref{eq:1.6}--\eqref{eq:1.8} are also defined analogously.
\end{definition}

One can prove the well-posedness of \eqref{eq:1.1}--\eqref{eq:1.3} in the sense of Definition \ref{D:sol} (see, e.g.,~\cite{G:EnSol},~\cite{HB3},~\cite{Vazquez}), and moreover, one can also derive $\partial_t (|u|^{(q-2)/2}u) \in L^2(0,\infty;L^2(\Omega))$ and 
\begin{align*}
\MoveEqLeft
\frac{4(q-1)}{q^2} \int^{t_2}_{t_1} \left\|\partial_t (|u|^{(q-2)/2}u)(\tau)\right\|_{L^2(\Omega)}^2 \, \d \tau + \frac 12 \|\nabla u(t_2)\|_{L^2(\Omega)}^2 \\
&\leq \frac 12 \|\nabla u(t_1)\|_{L^2(\Omega)}^2
\end{align*}
for $0 \leq t_1 < t_2 < +\infty$, and hence,
\begin{align}
&u \in C([0,\infty);L^q(\Omega)) \cap C_{\rm weak}([0,\infty);H^1_0(\Omega)) \cap C_+([0,\infty);H^1_0(\Omega)), \label{apdx1}\\
&\partial_t (|u|^{q-2}u) \in C_+([0,\infty);H^{-1}(\Omega)) \label{apdx2}
\end{align}
(see~\cite[Appendix]{A16} for more details). Here $C_{\rm weak}$ and $C_+$ stand for the sets of all weakly continuous and strongly right-continuous (vector-valued) functions, respectively. The same regularity as above can also be proved for energy solutions to \eqref{eq:1.6}--\eqref{eq:1.8}. As for nonnegative solutions, their positivity and classical regularity in smooth domains are proved in~\cite{DKV91,JiXi23,JiXi22+}; on the other hand, there seems almost no regularity result beyond the energy framework for possibly sign-changing solutions. Moreover, the extinction time $t_* = t_*(u_0)$ is uniquely determined for each initial datum $u_0$. Estimates \eqref{ext_time} and that with the $H^1_0(\Omega)$-norm replaced by the $L^q(\Omega)$-norm can be proved (see, e.g.,~\cite{AK13}).

We next recall the notions of stability and instability for asymptotic profiles introduced in~\cite{AK13}. Here we emphasize again that the set $\mathcal X$ is used as the phase set for the dynamical system generated by the Cauchy-Dirichlet problem \eqref{eq:1.6}--\eqref{eq:1.8}.

\begin{definition}[Stability and instability of asymptotic profiles (cf.~\cite{AK13})]
\label{D:stbl}
{\ }Let $\phi$ be an asymptotic profile of an energy solution to \eqref{eq:1.1}--\eqref{eq:1.3} {\rm (}equivalently, a nontrivial solution to \eqref{eq:1.10}, \eqref{eq:1.11}{\rm )}.
\begin{enumerate}
\item $\phi$ is said to be \emph{stable}, if
      for any $\varepsilon>0$ there exists $\delta > 0$ such
      that any energy solution $v$ of \eqref{eq:1.6}, \eqref{eq:1.7} satisfies
	       \begin{equation*}
		\sup_{s \in [0, \infty)}
		 \|v(s) -\phi\|_{H^1_0(\Omega)} < \varepsilon,
	       \end{equation*}
	     whenever $v(0) \in \mathcal X$ and $\|v(0) - \phi\|_{H^1_0(\Omega)}<\delta$.
\item $\phi$ is said to be \emph{unstable}, if $\phi$ is not stable.
\item $\phi$ is said to be \emph{asymptotically stable}, if $\phi$ is stable,
      and moreover, there exists $\delta_0 > 0$
      such that any energy solution $v$ of \eqref{eq:1.6}, \eqref{eq:1.7}
      satisfies
      \begin{equation*}
       \lim_{s \nearrow \infty}\|v(s) - \phi\|_{H^1_0(\Omega)} = 0,
      \end{equation*}
      whenever $v(0) \in \mathcal X$ and $\|v(0) - \phi\|_{H^1_0(\Omega)}<\delta_0$.
\item[(iv)] $\phi$ is said to be \emph{exponentially stable}, if $\phi$ is stable, and moreover, there exist constants $C, \mu, \delta_1 > 0$ such that any energy solution $v$ of \eqref{eq:1.6}, \eqref{eq:1.7} satisfies
$$
\|v(s)-\phi\|_{H^1_0(\Omega)} \leq C \e^{-\mu s} \quad \mbox{ for all } \ s \geq 0,
$$
provided that $v(0) \in \mathcal X$ and $\|v(0)-\phi\|_{H^1_0(\Omega)} < \delta_1$.
\end{enumerate}
\end{definition}

Finally, let us briefly summarize a couple of stability results obtained in~\cite[Theorems 2 and 3]{AK13}:
\begin{enumerate}
 \item Every \emph{least-energy} solution to \eqref{eq:1.10}, \eqref{eq:1.11} is \emph{asymptotically stable} in the sense of Definition \ref{D:stbl}, provided that it is isolated in $H^1_0(\Omega)$ from all the other nontrivial solutions.
 \item Every \emph{sign-changing} solution to \eqref{eq:1.10}, \eqref{eq:1.11} is \emph{not} asymptotically stable in the sense of Definition \ref{D:stbl}. In addition, if it is isolated in $H^1_0(\Omega)$ from all the other nontrivial solutions, then it is \emph{unstable}.
\end{enumerate}
We also refer the interested reader to~\cite{INdAM13,AK14,A16}.

\section{Proofs of Theorem \ref{T:sc-conv} and Corollary \ref{C:stbl}}\label{S:sc-conv}

Let $v = v(x,s)$ be a (possibly sign-changing) energy solution to \eqref{eq:1.6}--\eqref{eq:1.8} such that $v(s_n) \to \phi$ strongly in $H^1_0(\Omega)$ for some sequence $s_n \to +\infty$ and \emph{nondegenerate} nontrivial solution $\phi = \phi(x)$ to \eqref{eq:1.10}, \eqref{eq:1.11}. Then we can verify that
\begin{equation}\label{H10-conv}
v(s) \to \phi \quad \mbox{ strongly in } H^1_0(\Omega) \ \mbox{ as } \ s \to +\infty
\end{equation}
(see~\cite[\S 2]{A21} for a proof).

Before proceeding to a proof, we briefly give an idea of proof in view of comparison with~\cite{A21}, where a proof starts with the energy inequality,
$$
\frac{4(q-1)}{q^2} \left\|\partial_s (|v|^{(q-2)/2}v)(s)\right\|_{L^2(\Omega)}^2 \leq - \dfrac{\d}{\d s} J(v(s)).
$$
Observing the fundamental relation,
$$
\partial_s (|v|^{q-2}v)(s) = \frac{2(q-1)}q |v(s)|^{(q-2)/2} \partial_s (|v|^{(q-2)/2}v)(s),
$$
we can rewrite the left-hand side of the energy inequality as follows:
\begin{align*}
\MoveEqLeft{
\frac{4(q-1)}{q^2} \left\|\partial_s (|v|^{(q-2)/2}v)(s)\right\|_{L^2(\Omega)}^2
}\\
&= \frac 1 {q-1} \int_\Omega |\partial_s (|v|^{q-2}v)(x,s)|^2 |v(x,s)|^{2-q} \, \d x.
\end{align*}
In~\cite{A21}, for the case where $\phi$ is a positive solution to \eqref{eq:1.10}, \eqref{eq:1.11} in $\Omega$, in order to control the singularity arising from $|v(x,s)|^{2-q}$ (on the boundary; indeed, $v = 0$ on $\partial \Omega$ and $2-q < 0$), we substitute the profile $\phi(x) > 0$ in a proper way and rewrite the energy inequality as follows:
\begin{align*}
\MoveEqLeft{
\frac 1 {q-1} \left\|\frac{v(s)}{\phi}\right\|_{L^\infty(\Omega)}^{2-q} \int_\Omega |\partial_s (|v|^{q-2}v)(x,s)|^2 \phi(x)^{2-q} \, \d x
}\\
&\leq -\frac{\d}{\d s} J(v(s)).
\end{align*}
Then the ratio $\frac{v(s)}{\phi}$ is known to converge to 1 uniformly on $\Omega$ by~\cite{BGV12} (see also~\cite{BF21} for quantitative convergence), and moreover, $\partial_s (|v|^{q-2}v)(s)$ coincides with $-J'(v(s))$ by equation. As in~\cite{A21}, developing a quantitative gradient inequality, i.e., a relation between the energy gap $J(w)-J(\phi)$ and the weighted $L^2$-norm $\|J'(w)\|_{L^2(\Omega;\phi^{2-q}\, \d x)}$ of the gradient $J'(w)$, we can eventually obtain \eqref{h10-conv}. On the other hand, in the present paper, we shall directly handle the integral
$$
\int_\Omega |\partial_s (|v|^{q-2}v)(x,s)|^2 |v(x,s)|^{2-q} \, \d x
$$
as a weighted $L^2$-norm with the \emph{dynamic} weight function $|v(x,s)|^{2-q}$, which has singularity on the zero set of $v(\cdot,s)$ and may vary in time, and develop a quantitative gradient inequality for such time-dependent weighted $L^2$-norms of the gradients. To this end, we shall first carefully set up appropriate function spaces in the next subsection. A modified energy inequality will then be given in \S \ref{Ss:EI}. Next, we shall consider the eigenvalue problem for some linearized operator $\Ls$ at $v(s)$ associated with the evolutionary problem \eqref{eq:1.6}, \eqref{eq:1.7} in \S \ref{Ss:As}, and then, quantitative convergence as $s \to \infty$ of eigenvalues $\mu_j^s$ for $\Ls$ will also be discussed in \S \ref{Ss:ev-conv}. Furthermore, a quantitative gradient inequality will be developed under such a spectral framework in \S \ref{Ss:GI} based on some preparatory steps \S \ref{Ss:dual-decomp} and \S \ref{Ss:Taylor}. Finally, Theorem \ref{T:sc-conv} and Corollary \ref{C:stbl} will be proved at the end of \S \ref{Ss:SRC}.

On the other hand, an alternative argument will also be exhibited in Section \ref{S:alt}, where an ``$\vep$-approximation'' of the time-dependent $L^2$-norm will be introduced.

\subsection{$L^2$-spaces with possibly degenerate weights}\label{Ss:wL2}

In this subsection, we shall introduce $L^2$ spaces with possibly degenerate weights and their associate spaces. They will play a fundamental role in what follows. Moreover, we shall also discuss embeddings associated with these spaces. 

Let $s \geq 0$ be fixed and define the set of zeros of $v(\cdot,s)$ as
$$
Z(s) := \left\{ x \in \Omega \colon v(x,s) = 0\right\}.
$$
We set
$$
\H_s := \left\{ w \in \mathcal{M}(\Omega\setminus Z(s)) \colon |v(s)|^{q-2}w^2 \in L^1(\Omega\setminus Z(s))\right\},
$$
where $\mathcal{M}(\Omega \setminus Z(s))$ stands for the set of Lebesgue measurable functions defined on $\Omega \setminus Z(s)$, endowed with the inner product
$$
(f,g)_{\H_s} := \int_{\Omega\setminus Z(s)} f(x) g(x) |v(x,s)|^{q-2} \, \d x \quad \mbox{ for } \ f,g \in \H_s.
$$
Then $\H_s$ is a Hilbert space, whose norm is given by
$$
\|f\|_{\H_s}^2 = \int_{\Omega \setminus Z(s)} |f(x)|^2 \, |v(x,s)|^{q-2} \, \d x \quad \mbox{ for } \ f \in \H_s.
$$
Indeed, let $(f_n)$ be a Cauchy sequence in $\H_s$, i.e., $\|f_m - f_n\|_{\H_s} \to 0$ as $m,n \to +\infty$. Then $(f_n|v(s)|^{(q-2)/2})$ forms a Cauchy sequence in $L^2(\Omega\setminus Z(s))$. Hence it converges to a limit $h$ strongly in $L^2(\Omega\setminus Z(s))$. Set $f := h |v(s)|^{(2-q)/2}$. Then $f$ belongs to $\H_s$ and $f_n \to f$ strongly in $\H_s$. Hence $(\H_s,\|\cdot\|_{\H_s})$ is complete.

\begin{proposition}[Associate space of $\H_s$]\label{P:assp}
For each $s \geq 0$, the associate space $\H_s'$ of $\H_s$ is characterized as a Hilbert space,
\begin{equation}\label{assp}
\H_s' = \left\{ w \in \mathcal{M}(\Omega\setminus Z(s)) \colon 
|v(s)|^{2-q}w^2 \in L^1(\Omega\setminus Z(s)) \right\}
\end{equation}
equipped with the inner product
\begin{equation}\label{ip-H'}
(f,g)_{\H_s'} := \int_{\Omega \setminus Z(s)} f(x) g(x) |v(x,s)|^{2-q} \, \d x \quad \mbox{ for } \ f,g \in \H_s'.
\end{equation}
\end{proposition}

\begin{proof}
Let $s \geq 0$ be fixed. The associate space $\H_s'$ of $\H_s$ is defined by
$$
\H_s' := \left\{ f \in \mathcal{M}(\Omega\setminus Z(s)) \colon \|f\|_{\H_s'} < +\infty \right\}
$$
equipped with the norm
$$
\|f\|_{\H_s'} := \sup_{\substack{g \in \H_s\\\|g\|_{\H_s} \leq 1}} \int_{\Omega \setminus Z(s)} |f(x)||g(x)| \, \d x \quad \mbox{ for } \ f \in \mathcal{M}(\Omega \setminus Z(s)).
$$
Let $f \in \H_s'$ be fixed. Then we observe by definition that
\begin{align*}
\|f\|_{\H_s'}
&= \sup_{\substack{h \in L^2(\Omega\setminus Z(s))\\
\|h\|_{L^2(\Omega\setminus Z(s))} \leq 1}} \int_{\Omega\setminus Z(s)} |f(x)||h(x)||v(x,s)|^{(2-q)/2} \, \d x\\
&= \left\| f |v(s)|^{(2-q)/2}\right\|_{L^2(\Omega\setminus Z(s))}
\end{align*}
(here we set $h = g |v(s)|^{(q-2)/2} \in L^2(\Omega \setminus Z(s))$). Hence $f |v(s)|^{(2-q)/2}$ lies on $L^2(\Omega\setminus Z(s))$. The inverse also follows immediately as above. Furthermore, one can easily check that $(\cdot,\cdot)_{\H_s'}$ defined by \eqref{ip-H'} turns out to be the inner product which induces the norm $\|\cdot\|_{\H_s'}$, that is, $(f,f)_{\H_s'} = \|f\|_{\H_s'}^2$ for $f \in \H_s'$. Finally, the completeness of $(\H_s',\|\cdot\|_{\H_s'})$ can be checked similarly to $\H_s$.
\end{proof}

Due to the difference of domains, even if $u$ is the zero element of $\H_s$ or $\H_s'$ (i.e., $u = 0$ in $\H_s$ or in $\H_s'$), we cannot always assure that $u = 0$ a.e.~in $\Omega$ (but it is still true that $u = 0$ a.e.~in $\Omega \setminus Z(s)$).

\begin{proposition}\label{P:embed}
There exists a constant $C_* \geq 0$ depending on the supremum 
\begin{equation}\label{vLq}
c(v) := \sup_{s \geq 0} \|v(s)\|_{L^q(\Omega)} < +\infty
\end{equation}
such that
$$
\|w|_{\Omega\setminus Z(s)}\|_{\H_s} \leq C_* \|w\|_{L^q(\Omega)} \quad \mbox{ for } \ w \in L^q(\Omega) \ \mbox{ and } \ s \geq 0,
$$
where $w|_{\Omega \setminus Z(s)}$ stands for the restriction of $w$ onto $\Omega \setminus Z(s)$ and will be denoted simply by $w$ when no confusion can arise. Moreover, let $f \in \H_s'$ and denote by $\bar f$ the zero extension of $f$ onto $\Omega$. Then $\bar f$ belongs to $L^{q'}(\Omega)$, and moreover, it holds that
$$
\|\bar f\|_{L^{q'}(\Omega)} \leq C_* \|f\|_{\H_s'} \quad \mbox{ for } \ f \in \H_s' \ \mbox{ and } \ s \geq 0.
$$
\end{proposition}

\begin{proof}
For each $s \geq 0$, we observe that
$$
\int_{\Omega \setminus Z(s)} |w|^2 |v(s)|^{q-2} \, \d x \leq \|w\|_{L^q(\Omega)}^2 \|v(s)\|_{L^q(\Omega)}^{q-2} \quad \mbox{ for } \ w \in L^q(\Omega).
$$
Since $v(s)$ is bounded in $L^q(\Omega)$ for $s \geq 0$ (see \eqref{v-bdd}), there exists a constant $C_* \geq 0$ depending on \eqref{vLq} (e.g., one can take $C_* = c(v)^{(q-2)/2}$) such that
$$
\|w|_{\Omega \setminus Z(s)}\|_{\H_s} \leq C_* \|w\|_{L^q(\Omega)} \quad \mbox{ for } \ w \in L^q(\Omega) \ \mbox{ and } \ s \geq 0.
$$
Next, let $f \in \H_s'$ and define $\bar f : \Omega \to \R$ by $\bar f(x) = f(x)$ if $x \in \Omega \setminus Z(s)$; $\bar f(x) = 0$ if $x \in Z(s)$. It then follows that
\begin{align*}
\left| \int_\Omega \bar f \varphi \, \d x \right|
\leq \int_{\Omega \setminus Z(s)} |f \varphi| \, \d x 
&\leq \|f\|_{\H_s'} \|\varphi|_{\Omega \setminus Z(s)}\|_{\H_s}\\
&\leq C_* \|f\|_{\H_s'} \|\varphi\|_{L^q(\Omega)} \quad \mbox{ for } \ \varphi \in L^q(\Omega),
\end{align*}
which along with the Riesz representation theorem implies that $\bar f \in L^{q'}(\Omega)$, and hence, we obtain $\|\bar f\|_{L^{q'}(\Omega)} \leq C_* \|f\|_{\H_s'}$ for $f \in \H_s'$ and $s \geq 0$.
\end{proof}

\subsection{A modified energy inequality}\label{Ss:EI}

We next derive some energy inequality for \eqref{eq:1.6}--\eqref{eq:1.8} by employing the family of Hilbert spaces introduced in the former section. Define a functional $J : H^1_0(\Omega) \to \R$ by
$$
J(w) = \frac 12\|\nabla w\|_{L^2(\Omega)}^2 - \frac{\lambda_q}{q} \|w\|_{L^q(\Omega)}^q \ \mbox{ for } \ w \in H^1_0(\Omega).
$$
Then one has $\partial_s (|v|^{(q-2)/2})v \in L^2(0,\infty;L^2(\Omega))$ and
$$
\frac{4(q-1)}{q^2} \int^{s_2}_{s_1} \left\|\partial_s (|v|^{(q-2)/2}v)(\sigma)\right\|_{L^2(\Omega)}^2 \, \d \sigma + J(v(s_2)) \leq J(v(s_1))
$$
for $0 \leq s_1 < s_2 < +\infty$; whence it follows that the function $s \mapsto J(v(s))$ is nonincreasing and hence differentiable a.e.~in $(0,+\infty)$. It also follows that
\begin{equation}\label{eneq}
 \frac{4(q-1)}{q^2} \left\|\partial_s (|v|^{(q-2)/2}v)(s)\right\|_{L^2(\Omega)}^2 \leq - \dfrac{\d}{\d s} J(v(s))
\end{equation}
for a.e.~$s > 0$.

Noting that
\begin{equation}\label{time-deri-decomp}
\partial_s (|v|^{q-2}v)(s) = \frac{2(q-1)}q |v(s)|^{(q-2)/2} \partial_s (|v|^{(q-2)/2}v)(s),
\end{equation}
which also implies that $\partial_s (|v|^{q-2}v) \in L^2(0,\infty;L^{q'}(\Omega))$, and recalling that $\partial_s (|v|^{(q-2)/2}v)(s) \in L^2(\Omega)$, we can deduce that $\partial_s (|v|^{q-2}v)(s) \in \H_s'$ for a.e.~$s > 0$. For $w \in \H_s$, we observe that
\begin{align*}
\MoveEqLeft
 \int_{\Omega \setminus Z(s)} \left| \partial_s (|v|^{q-2}v)(s) \right| |w|  \, \d x
\\
&= \frac{2(q-1)}{q} \int_{\Omega \setminus Z(s)} |v(s)|^{(q-2)/2} \left| \partial_s (|v|^{(q-2)/2}v)(s) \right| |w| \, \d x\\
&\leq \frac{2(q-1)}{q} \left\|\partial_s (|v|^{(q-2)/2}v)(s)\right\|_{L^2(\Omega)} \left\|w\right\|_{\H_s},
\end{align*}
which implies
\begin{equation}\label{ds-ineq}
 \left\|\partial_s (|v|^{q-2}v)(s)\right\|_{\H_s'}
  \leq \frac{2(q-1)}{q} \left\|\partial_s (|v|^{(q-2)/2}v)(s)\right\|_{L^2(\Omega)} 
\end{equation}
for a.e.~$s > 0$. Recalling the energy inequality \eqref{eneq}, we obtain
\begin{equation}\label{eneq2}
\frac 1{q-1} \left\| \partial_s (|v|^{q-2}v)(s) \right\|_{\H_s'}^2
\leq - \dfrac{\d}{\d s} J(v(s)).
\end{equation}

\subsection{Eigenvalue problems with possibly degenerate weights}\label{Ss:As}

Throughout this subsection, let $s \geq 0$ be fixed arbitrarily. Define the operator $A^s : H^1_0(\Omega) \to H^1_0(\Omega)$ by
$$
A^s(w) := (-\Delta)^{-1} \left( |v(s)|^{q-2} w \right) \quad \mbox{ for } \ w \in H^1_0(\Omega).
$$
The following argument in this subsection is still valid for any function $\omega \in H^1_0(\Omega) \setminus \{0\}$ instead of $v(s)$ in the weight (e.g., $\omega = \phi$). Then $\mathcal{H}_s$ is also replaced in an analogous way, and then, it does no longer depend on $s$. This subsection is devoted to discussing eigenvalue problems for the operator $A^s$. We shall finally construct a complete orthonormal system of $H^1_0(\Omega)$ by means of eigenfunctions for $A^s$.

We first prove that $A^s$ is self-adjoint and compact. Indeed, let $(w_n)$ be bounded in $H^1_0(\Omega)$ and set $u_n := A^s(w_n)$. Then $- \Delta u_n = |v(s)|^{q-2}w_n$ in $H^{-1}(\Omega)$. Testing both sides by $u_n$, we see that
\begin{align*}
\|\nabla u_n\|_{L^2(\Omega)}^2 &= \int_\Omega |v(s)|^{q-2} w_n u_n \, \d x\\
&\leq C_q^2 \|v(s)\|_{L^q(\Omega)}^{q-2} \|w_n\|_{H^1_0(\Omega)} \|\nabla u_n\|_{L^2(\Omega)},
\end{align*}
which implies that $(u_n)$ is bounded in $H^1_0(\Omega)$. Hence we have, up to a (not relabeled) subsequence, $u_n \to u$ and $w_n \to w$ weakly in $H^1_0(\Omega)$ and strongly in $L^q(\Omega)$ (by $q < 2^*$) for some $u,w \in H^1_0(\Omega)$. Therefore one can verify that $-\Delta u = |v(s)|^{q-2} w$ in $H^{-1}(\Omega)$, and moreover, it follows that
\begin{align*}
 \|\nabla u_n\|_{L^2(\Omega)}^2
 &= \int_\Omega |v(s)|^{q-2} w_n u_n \, \d x \\
 &\to \int_\Omega |v(s)|^{q-2} w u \, \d x = \|\nabla u\|_{L^2(\Omega)}^2,
\end{align*}
which along with the uniform convexity of $(H^1_0(\Omega), \|\nabla \cdot\|_{L^2(\Omega)})$ yields
$$
u_n \to u \quad \mbox{ strongly in } H^1_0(\Omega).
$$
Thus $A^s$ turns out to be a compact operator in $H^1_0(\Omega)$. Moreover, let $f,g \in H^1_0(\Omega)$. It then follows immediately that
\begin{align*}
(A^s f, g)_{H^1_0(\Omega)} &= \langle -\Delta (A^s f), g \rangle_{H^1_0(\Omega)}
= \int_\Omega |v(s)|^{q-2} f g \, \d x\\
& = \langle -\Delta (A^s g), f \rangle_{H^1_0(\Omega)}
= (f, A^s g)_{H^1_0(\Omega)},
\end{align*}
where $(\cdot,\cdot)_{H^1_0(\Omega)}$ and $\langle\cdot,\cdot\rangle_{H^1_0(\Omega)}$ stand for the inner product in $H^1_0(\Omega)$ and the duality pairing between $H^1_0(\Omega)$ and $H^{-1}(\Omega)$, respectively (see Notation). Hence $A^s$ is symmetric, and therefore, self-adjoint.

Due to the spectral theory for compact self-adjoint operators (see, e.g.,~\cite[\S 5]{B-FA}), all eigenvalues of $A^s$ are real and bounded. Moreover, since $A^s \neq 0$, $\sigma(A^s) \setminus \{0\}$ is either finite or a sequence converging to $0$. Here $\sigma(A^s)$ stands for the spectral set of $A^s$. Furthermore, $\sigma(A^s) \setminus \{0\}$ coincides with the set of all nonzero eigenvalues of $A^s$. Moreover, we observe that all eigenvalues of $A^s$ are nonnegative. \prf{Indeed, let $(\lambda, e)$ be an eigenpair of $A^s$ such that $\|e\|_{H^1_0(\Omega)} = 1$. Testing both sides of $A^s e = \lambda e$ by $e$, we have
$$
\langle -\Delta (A^s e), e \rangle_{H^1_0(\Omega)} = (A^s e, e)_{H^1_0(\Omega)} = \lambda \|e\|_{H^1_0(\Omega)}^2 = \lambda.
$$
We also note that
$$
\langle -\Delta (A^s e), e \rangle_{H^1_0(\Omega)} = \int_\Omega |v(s)|^{q-2} e^2 \, \d x \geq 0;
$$
whence it follows that $\lambda \geq 0$.}

Let $\{\lambda_j^s\}_{j \geq 1}$ be the set of all \emph{nonzero} eigenvalues of $A^s$. We set
\begin{align*}
E_0^s := N(A^s) :=& \ \{w \in H^1_0(\Omega) \colon A^s(w) = 0\}\\
=& \ \{w \in H^1_0(\Omega) \colon |v(s)|^{q-2}w = 0 \ \mbox{ a.e.~in } \Omega\}
\end{align*}
and $E_j^s := N(A^s - \lambda_j^s I)$. Then we find that
$$
0 \leq \dim E_0^s \leq +\infty, \quad 0 < \dim E_j^s < +\infty.
$$
Here the latter follows from the Fredholm alternative. Then $H^1_0(\Omega)$ is the Hilbert sum of $\{E_j^s\}_{j \geq 0}$.

We claim that $\mathrm{codim} \, E_0^s = +\infty$; indeed, since $E^s_0 = N(A^s)$ is closed in $H^1_0(\Omega)$, we find that $\mathrm{codim} \, E^s_0 < +\infty$ if and only if $\dim \, (E^s_0)^\perp < +\infty$ (see, e.g.,~\cite[(b) of Proposition 11.13]{B-FA}). Recall that
$$
(E^s_0)^\perp = \{ f \in H^{-1}(\Omega) \colon \langle f, \varphi \rangle_{H^1_0(\Omega)} = 0 \ \mbox{ for } \ \varphi \in E^s_0 \}.
$$
Since the $N$-dimensional Lebesgue measure $\mathcal{L}^N(\Omega \setminus Z(s))$ of $\Omega \setminus Z(s)$ is positive (otherwise, $v(s) \equiv 0$), one can construct a sequence $\{M_j\}_{j=1}^\infty$ of disjoint Lebesgue measurable sets such that $\mathcal{L}^N(M_j) > 0$ for $j \in \N$ and $\cup_{j=1}^\infty M_j = \Omega \setminus Z(s)$ (indeed, it is possible, e.g., since $r \mapsto \mathcal{L}^N((\Omega \setminus Z(s)) \cap B_r)$ is continuous for a ball $B_r$ in $\R^N$ of radius $r$). Each characteristic function $\chi_{M_j}$ supported over $M_j$ belongs to $H^{-1}(\Omega)$. Moreover,  for $j \in \N$, noting that $M_j \subset \Omega \setminus Z(s)$ and $\varphi = 0$ a.e.~in $\Omega \setminus Z(s)$ for $\varphi \in E^s_0$, we observe that
$$
\langle \chi_{M_j}, \varphi \rangle_{H^1_0(\Omega)} = \int_\Omega \chi_{M_j} \varphi \, \d x = 0 \quad \mbox{ for } \ \varphi \in E^s_0. 
$$
Hence $\chi_{M_j}$ lies on $(E^s_0)^\perp$ for $j \in \N$. Thus we obtain $\dim \,(E^s_0)^\perp = +\infty$.

Hence $\sigma(A^s) \setminus \{0\} = \{\lambda_j^s\}_{j\geq 1}$ turns out to be a sequence converging to $0$. Here and henceforth, we denote by $\{(\lambda^s_j, e^s_j)\}_{j=1}^\infty$ the sequence consisting of all eigenpairs of $A^s$ for nonzero eigenvalues such that $\{\lambda^s_j\}_{j=1}^\infty$ is nonincreasing, $\lambda^s_j \to 0$ as $j \to +\infty$ and $e^s_j$ is an eigenfunction corresponding to the eigenvalue $\lambda^s_j$ and normalized in $H^1_0(\Omega)$ (that is, $\|e^s_j\|_{H^1_0(\Omega)} = 1$ and then $\|e_j^s/\sqrt{\lambda_j^s}\|_{\H_s} = 1$) by rearranging eigenvalues and by repeating the same eigenvalue according to its multiplicity.

Set $\mu^s_j := 1/\lambda^s_j > 0$ ($j \geq 1$). Then for each $j\geq 1$, $(\mu^s_j,e^s_j)$ is an eigenpair of the following eigenvalue problem:
\begin{equation}\label{eps}
 - \Delta e = \mu |v(s)|^{q-2} e \ \mbox{ in } \Omega, \quad e = 0 \ \mbox{ on } \partial \Omega.
\end{equation}
Then, for any $u \in H^1_0(\Omega)$, there exists $u_0^s \in E_0^s$ such that
\begin{equation}\label{u-ext}
u = u_0^s + \sum_{j=1}^\infty \alpha^s_j e^s_j \ \mbox{ in } H^1_0(\Omega), \quad \alpha^s_j := (u,e^s_j)_{H^1_0(\Omega)},
\end{equation}
which implies that
\begin{align}
 \Ls u :=& - \Delta u - \lambda_q (q-1)|v(s)|^{q-2} u \nonumber\\
 =& -\Delta u_0^s + \sum_{j=1}^\infty \alpha^s_j \frac{\mu_j^s - \lambda_q (q-1)}{\mu_j^s} (-\Delta) e^s_j \ \mbox{ in } H^{-1}(\Omega).\label{Lu-ext}
\end{align}

Let us also consider the eigenvalue problem,
\begin{equation}\label{ep}
 - \Delta e = \mu |\phi|^{q-2} e \ \mbox{ in } \Omega, \quad e = 0 \ \mbox{ on } \partial \Omega,
\end{equation}
where $\phi \in H^1_0(\Omega)$ is a (possibly sign-changing) solution to \eqref{eq:1.10}, \eqref{eq:1.11}. Then repeating the argument so far, we can construct eigenpairs $\{(\mu_j,e_j)\}_{j \geq 1}$ of \eqref{ep} for positive eigenvalues. Then $(\nu_j,e_j)$ with $\nu_j := \mu_j - \lambda_q (q-1)$ becomes an eigenpair of the linearized operator 
$$
\Lphi := - \Delta - \lambda_q (q-1) |\phi|^{q-1}
$$
(cf.~see~\cite[\S 2]{BF21},~\cite{A21}). When $\phi$ is nondegenerate, all the eigenvalues $\nu_j$ are nonzero, and therefore, $\Lphi \in \mathscr{L}(H^1_0(\Omega),H^{-1}(\Omega))$ is invertible (i.e., the inverse $\Lphi^{-1} :H^{-1}(\Omega) \to H^1_0(\Omega)$ is well defined and bounded linear). Moreover, $\mathcal{H}_s$ and $\mathcal{H}_s'$ are replaced by $\mathcal{H}_\phi := L^2(\Omega \setminus Z_\phi;|\phi|^{q-2}\, \d x)$ and $\mathcal{H}_\phi' := L^2(\Omega \setminus Z_\phi;|\phi|^{2-q}\, \d x)$ with the set $Z_\phi := \{x \in \Omega \colon \phi(x) = 0\}$.

\subsection{Convergence of eigenvalues}\label{Ss:ev-conv}

In this subsection, we shall discuss convergence of each eigenvalue $\mu^s_j$ for \eqref{eps} as $s \to +\infty$. We first exhibit the following lemma, which may be standard; however, for the completeness, we shall give a proof in Appendix \S \ref{A:ev-var}:

\begin{lemma}[Variational representation of eigenvalues]\label{L:ev-var}
For each $s \geq 0$ and $j \geq 1$, the eigenvalue $\mu^s_j > 0$ of \eqref{eps} can be characterized as the following max-min value\/{\rm :}
\begin{align}
\frac 1 {\mu^s_j} &= \sup_{\substack{Y \subset H^1_0(\Omega)\\[0mm]\dim Y = j}} \inf_{\substack{w \in Y\\[0mm]\|w\|_{H^1_0(\Omega)}=1}} (A^s w,w)_{H^1_0(\Omega)}\nonumber\\
&= \sup_{\substack{Y \subset H^1_0(\Omega)\\[0mm]\dim Y = j}} \inf_{\substack{w \in Y\\[0mm]\|w\|_{H^1_0(\Omega)}=1}} \int_\Omega |v(s)|^{q-2} w^2 \, \d x.\label{ev-var-s}
\end{align}
Here $Y$ denotes a subspace of $H^1_0(\Omega)$. Similarly, each eigenvalue $\mu_j > 0$ of \eqref{ep} can be written as
\begin{equation}\label{ev-var}
\frac 1 {\mu_j} = \sup_{\substack{Y \subset H^1_0(\Omega)\\[0mm]\dim Y = j}} \inf_{\substack{w \in Y\\[0mm]\|w\|_{H^1_0(\Omega)}=1}} \int_\Omega |\phi|^{q-2} w^2 \, \d x.
\end{equation} 
\end{lemma}

Now, we are ready to prove
\begin{lemma}[Convergence of eigenvalues]\label{L:ev-pert}
There exists a positive constant $C$ which depends only on $q$, $C_q$ and $c(v)$ defined by \eqref{vLq} such that
\begin{equation}\label{ev-pert}
\left|\frac 1 {\mu^s_j} - \frac 1 {\mu_j}\right| \leq C \| v(s) - \phi \|_{L^q(\Omega)}^{\rho} \quad \mbox{ for } \ s \geq 0 \ \mbox{ and } \ j \in \N,
\end{equation}
where $\rho := \min\{q-2,1\} \in (0,1]$. Moreover, for each $j \in \N$, it holds that
\begin{equation}\label{conv-ev}
 \mu^s_j \to \mu_j \quad \mbox{ as } \ s \to +\infty.
\end{equation}
\end{lemma}

\begin{proof}
Note that
\begin{align*}
\MoveEqLeft
 \int_\Omega |v(s)|^{q-2} w^2 \, \d x
\\
&\leq \int_\Omega |\phi|^{q-2} w^2 \, \d x
+ \int_\Omega \left| |v(s)|^{q-2} - |\phi|^{q-2} \right| w^2 \, \d x\\
&\leq \int_\Omega |\phi|^{q-2} w^2 \, \d x
+ C_q^2 \left\| |v(s)|^{q-2} - |\phi|^{q-2} \right\|_{L^{q/(q-2)}(\Omega)} \|w\|_{H^1_0(\Omega)}^2
\end{align*}
for $w \in H^1_0(\Omega)$. Taking the sup-inf of both sides as in Lemma \ref{L:ev-var}, we obtain
\begin{equation}\label{ev-pert-1}
\frac 1 {\mu^s_j} \leq \frac 1 {\mu_j} + C_q^2 \left\| |v(s)|^{q-2} - |\phi|^{q-2} \right\|_{L^{q/(q-2)}(\Omega)}.
\end{equation}

In case $2 < q < 3$, it follows that
$$
\left||v(x,s)|^{q-2}-|\phi(x)|^{q-2}\right| \leq \left||v(x,s)|-|\phi(x)|\right|
^{q-2}\leq \left|v(x,s)-\phi(x)\right|^{q-2},
$$
which yields
$$
\left\| |v(s)|^{q-2} - |\phi|^{q-2} \right\|_{L^{q/(q-2)}(\Omega)} \leq \|v(s)-\phi\|_{L^q(\Omega)}^{q-2}.
$$
In case $q \geq 3$, we observe that
\begin{align*}
\lefteqn{
\left||v(x,s)|^{q-2}-|\phi(x)|^{q-2}\right|
}\\
&\leq (q-2) \left(|v(x,s)|^{q-3} + |\phi(x)|^{q-3} \right) |v(x,s)-\phi(x)|.
\end{align*}
Hence
\begin{align*}
\MoveEqLeft
\left\| |v(s)|^{q-2} - |\phi|^{q-2} \right\|_{L^{q/(q-2)}(\Omega)}
\\
&\leq (q-2) \left( \|v(s)\|_{L^q(\Omega)}^{q-3} + \|\phi\|_{L^q(\Omega)}^{q-3} \right)\|v(s)-\phi\|_{L^q(\Omega)}.
\end{align*}
Therefore, since $v(s)$ is bounded in $L^q(\Omega)$ for $s \geq 0$, we conclude that
\begin{equation}\label{power-diff}
\left\| |v(s)|^{q-2} - |\phi|^{q-2} \right\|_{L^{q/(q-2)}(\Omega)} \leq c \|v(s)-\phi\|_{L^q(\Omega)}^\rho,
\end{equation}
where $\rho := \min \{q-2,1\} \in (0,1]$, for some constant $c \geq 0$ which depends on $q$ and $c(v)$ defined by \eqref{vLq}.

Consequently, it follows from \eqref{ev-pert-1} that
$$
\frac 1 {\mu^s_j} \leq \frac 1 {\mu_j} + C \| v(s) - \phi \|_{L^q(\Omega)}^{\rho}
$$
for any $s \geq 0$ and $j \in \N$. Here $C$ depends only on $q$, $C_q$ and $c(v)$. One can also prove the inverse inequality,
\begin{equation}\label{sim}
\frac 1 {\mu^s_j} \geq \frac 1 {\mu_j} - C \| v(s) - \phi \|_{L^q(\Omega)}^{\rho}
\end{equation}
in a similar fashion. Thus we conclude that 
$$
\left| \frac 1 {\mu^s_j} - \frac 1 {\mu_j} \right| \leq C \| v(s) - \phi \|_{L^q(\Omega)}^{\rho} \quad \mbox{ for } \ s \geq 0,
$$
which further implies
\begin{equation}\label{end2}
\left| \mu^s_j - \mu_j \right| \leq C \mu^s_j \mu_j \| v(s) - \phi \|_{L^q(\Omega)}^{\rho} \quad \mbox{ for } \ s \geq 0.
\end{equation}
From \eqref{sim} along with the positivity of $\mu_j$ and \eqref{H10-conv}, we observe that $\mu^s_j \leq 2\mu_j$ for $s \geq 0$ large enough. Hence we obtain
$$
\left| \mu^s_j - \mu_j \right| \leq 2C \mu_j^2 \| v(s) - \phi \|_{L^q(\Omega)}^{\rho} \quad \mbox{ for } \ s \gg 1,
$$
which along with \eqref{H10-conv} yields \eqref{conv-ev}. This completes the proof.
\end{proof}

As a corollary, we have the following:

\begin{remark}[Invertibility of $\Ls$ for $s > 0$ large enough]\label{R:Ls-inv}
\begin{enumerate}
\item[(i)] Since $\phi$ is nondegenerate, that is, $\nu_j = \mu_j - \lambda_q(q-1) \neq 0$ for any $j \in \N$, we find from \eqref{conv-ev} that $\nu^s_j = \mu^s_j - \lambda_q(q-1) \neq 0$ for $s > 0$ large enough. In particular, we deduce that $\Ls \in \mathscr{L}(H^1_0(\Omega), H^{-1}(\Omega))$ is invertible with its inverse $\Ls^{-1} \in \mathscr{L}(H^{-1}(\Omega),H^1_0(\Omega))$ for $s > 0$ large enough. 
\item[(ii)] For each $s \geq 0$, let $k(s) \in \N$ be the least number such that $\nu^s_{k(s)} = \mu^s_{k(s)} - \lambda_q (q-1)$ is positive. Let $k \in \N$ be the least number such that $\nu_k = \mu_k - \lambda_q(q-1) > 0$. Since $\nu_j \neq 0$ for any $j \in \N$, we deduce from \eqref{conv-ev} that $k(s) = k$ for $s > 0$ large enough. Hence, in what follows, we shall simply write $k$ instead of $k(s)$ for $s > 0$ large enough.
\end{enumerate}
\end{remark}

Moreover, we claim that 
\begin{lemma}\label{L:Lsinv-bdd}
There exists a constant $s_0 \geq 0$ such that
\begin{equation}\label{L-bdd}
\sup_{s \geq s_0} \|\Ls^{-1}\|_{\mathscr{L}(H^{-1}(\Omega),H^1_0(\Omega))} \leq 2 \|\Lphi^{-1}\|_{\mathscr{L}(H^{-1}(\Omega),H^1_0(\Omega))}.
\end{equation}
\end{lemma}

\begin{proof}
We observe that
\begin{align}
\Ls w &= - \Delta w - \lambda_q (q-1) |v(s)|^{q-2} w\nonumber\\
&= \Lphi w - \lambda_q  (q-1) \left( |v(s)|^{q-2} - |\phi|^{q-2} \right) w\nonumber\\
&= \Lphi \left(
w - \lambda_q (q-1) \Lphi^{-1} \left[
\left( |v(s)|^{q-2} - |\phi|^{q-2} \right) w
\right]
\right)\nonumber\\
&=: \Lphi \left( T_s(w)\right) \quad \mbox{ for } \ w \in H^1_0(\Omega),\label{Ls-inv-rel}
\end{align}
where $T_s : H^1_0(\Omega) \to H^1_0(\Omega)$ is a bounded linear operator given by $T_s(w) = w - \lambda_q  (q-1) \Lphi^{-1} [(|v(s)|^{q-2} - |\phi|^{q-2}) w]$ for $w \in H^1_0(\Omega)$. Noting that
\begin{align*}
\MoveEqLeft{
\|\Lphi^{-1} [(|v(s)|^{q-2} - |\phi|^{q-2}) w] \|_{H^1_0(\Omega)}
}\\
&\leq \|\Lphi^{-1}\|_{\mathscr{L}(H^{-1}(\Omega), H^1_0(\Omega))} \|(|v(s)|^{q-2} - |\phi|^{q-2}) w \|_{H^{-1}(\Omega)}\\
&\leq C_q \|\Lphi^{-1}\|_{\mathscr{L}(H^{-1}(\Omega), H^1_0(\Omega))}
\||v(s)|^{q-2} - |\phi|^{q-2}\|_{L^{q/(q-2)}(\Omega)} \|w\|_{L^q(\Omega)}
\end{align*}
and recalling \eqref{power-diff} along with \eqref{H10-conv}, we can take $s_0 \geq 0$ large enough so that
$$
\lambda_q  (q-1) \|\Lphi^{-1} [(|v(s)|^{q-2} - |\phi|^{q-2}) w] \|_{H^1_0(\Omega)} \leq \frac 12 \|w\|_{H^1_0(\Omega)} \quad \mbox{ for } \ w \in H^1_0(\Omega)
$$
for all $s \geq s_0$. Hence $T_s$ turns out to be invertible such that $\|T_s^{-1}\|_{\mathscr{L}(\H^1_0(\Omega))} \leq 2$ for $s \geq s_0$. Therefore, thanks to \eqref{Ls-inv-rel}, we obtain
\begin{align*}
\|\Ls^{-1}\|_{\mathscr{L}(H^{-1}(\Omega),H^1_0(\Omega))} 
&\leq \|T_s^{-1}\|_{\mathscr{L}(H^1_0(\Omega))} \|\Lphi^{-1}\|_{\mathscr{L}(H^{-1}(\Omega),H^1_0(\Omega))}\\
&\leq 2 \|\Lphi^{-1}\|_{\mathscr{L}(H^{-1}(\Omega),H^1_0(\Omega))}
\quad \mbox{ for } \ s \geq s_0,
\end{align*}
which completes the proof.
\end{proof}

We close this subsection with the following:
\begin{corollary}\label{C:data_close}
Let $\phi \in \mathcal{X}$ be a nondegenerate \emph{least-energy} solution to \eqref{eq:1.10}, \eqref{eq:1.11}. Then for any $\vep > 0$ there exists a constant $r_\vep > 0$ {\rm (}independent of $v_0$ and $s${\rm )} such that
\begin{equation}
 |\mu^s_j - \mu_j| < \vep \quad \mbox{ for } \ s \geq 0 \ \mbox{ and } \ j \in \N,
\end{equation}
provided that $v_0 \in \mathcal{X}$ 
satisfies $\|v_0 - \phi\|_{H^1_0(\Omega)} < r_\vep$. Moreover, if $\vep > 0$ is small enough, it holds that $\nu^s_j \neq 0$ and $k(s) = 2$ for any $j \in \N$ and $s \geq 0$ and \eqref{L-bdd} holds with $s_0 = 0$ under the same assumption for $v_0$ {\rm (}cf.~see Remark {\rm \ref{R:Ls-inv}} and Lemma {\rm \ref{L:Lsinv-bdd})}.
\end{corollary}

\begin{proof}
Let $\vep > 0$ be fixed. In what follows, in addition to $v_0 \in \mathcal{X}$, we always assume that $\|v_0 - \phi\|_{H^1_0(\Omega)} < 1$. Then we can take a constant $M > 0$ such that $c(v) = \sup_{s \geq 0}\|v(s)\|_{L^q(\Omega)} \leq M$, where $v$ denotes the energy solution to \eqref{eq:1.6}--\eqref{eq:1.8} with the initial datum $v_0$, for any $v_0 \in \mathcal{X}$ satisfying $\|v_0 - \phi\|_{H^1_0(\Omega)} < 1$ (see \eqref{v-bdd} and~\cite[Lemma 2]{A16}). We emphasize that $M$ can be taken uniformly for $v_0$ and $s$ satisfying the assumption above. Moreover, we observe from \eqref{ev-pert} that $\mu^s_j \leq 2 \mu_j$ for any $s \geq 0$, provided that 
$$
\sup_{s \geq 0}\|v(s)-\phi\|_{L^q(\Omega)}^\rho < \frac 1{2C\mu_j},
$$
where $C$ can now be taken uniformly for $v_0$ (instead, it may depend on $M$). It then follows from \eqref{end2} that
$$
|\mu^s_j - \mu_j| \leq 2C\mu_j^2 \|v(s)-\phi\|_{L^q(\Omega)}^\rho \quad \mbox{ for } \ s \geq 0.
$$
Thanks to the asymptotic stability result in~\cite[Theorem 2]{AK13}, there exists $r_\vep \in (0,1)$ such that
$$
\sup_{s \geq 0}\|v(s)-\phi\|_{L^q(\Omega)}^\rho < \min \left\{ \frac 1{2C\mu_j}, \frac \vep {2C\mu_j^2}\right\} 
$$
and $v(s) \to \phi$ strongly in $H^1_0(\Omega)$ as $s \to +\infty$,
whenever $v_0 \in \mathcal{X}$ and $\|v_0 - \phi\|_{H^1_0(\Omega)} < r_\vep$. Therefore we then obtain $|\mu^s_j - \mu_j| < \vep$ for any $s \geq 0$ and $j \in \N$.
\end{proof}

\subsection{Decomposition of the dual space}\label{Ss:dual-decomp}

In this subsection, we shall introduce a complete orthonormal system of $H^{-1}(\Omega)$ by means of the eigenfunctions and a Riesz map, and moreover, we shall discuss a spectral decomposition of the inverse $\Ls^{-1}$ of the linearized operator $\Ls$. Furthermore, it will eventually be proved that eigenfunctions of \eqref{eps} for \emph{positive} eigenvalues form a complete orthonormal system of $\H_s'$.

Recall that $-\Delta : H^1_0(\Omega) \to H^{-1}(\Omega)$ is a Riesz map and set 
$$
F_0^s = -\Delta E_0^s := \left\{ - \Delta w \colon w \in E_0^s \right\}
$$
for $s \geq 0$. Then for any $f \in H^{-1}(\Omega)$ and $s \geq 0$, one can take $f^s_0 \in F^s_0$ such that
\begin{equation}\label{f}
f = f_0^s + \sum_{j=1}^\infty \beta^s_j (-\Delta) e^s_j \ \mbox{ in } H^{-1}(\Omega), \quad \beta^s_j := \langle f, e^s_j \rangle_{H^1_0(\Omega)}.
\end{equation}
In what follows, we denote by $P_{F^s_0} : H^{-1}(\Omega) \to F^s_0$ the orthogonal projection onto $F^s_0$, that is, $P_{F^s_0}(f) = f^s_0$. One can derive from \eqref{f} along with \eqref{u-ext} and \eqref{Lu-ext} that
\begin{equation}\label{PO}
\Ls^{-1} f = \Linv f_0^s + \sum_{j=1}^\infty \beta^s_j \frac{\mu_j^s}{\mu_j^s - \lambda_q (q-1)} e^s_j
\end{equation}
for $s > 0$ large enough (so that $\nu^s_j = \mu^s_j - \lambda_q(q-1) \neq 0$ for $j \in \N$; see (i) of Remark \ref{R:Ls-inv}). Thus we obtain
\begin{equation}\label{fLinvf_0}
\langle f, \Ls^{-1} f \rangle_{H^1_0(\Omega)}
= \|f_0^s\|_{H^{-1}(\Omega)}^2 + \sum_{j=1}^\infty (\beta^s_j)^2 \frac{\mu_j^s}{\mu_j^s - \lambda_q (q-1)}
\end{equation}
for $f \in H^{-1}(\Omega)$ represented as \eqref{f} and $s > 0$ large enough.

We next have

\begin{lemma}\label{L:fs0=0}
For $f \in H^{-1}(\Omega)$ and $s \geq 0$, it holds that $P_{F^s_0}f = 0$ if and only if $f \in (E^s_0)^\perp$. In particular, if $f \in \H_s'$, then $P_{F^s_0} (\bar f) = 0$, where $\bar f$ is the zero extension of $f$ onto $\Omega$.
\end{lemma}

\begin{proof}
Fix $s \geq 0$. Let $f \in (E^s_0)^\perp \subset H^{-1}(\Omega)$ and let $f^s_0 = P_{F^s_0}(f)$. Then for any $\varphi \in E^s_0$, we see that
$$
(f,-\Delta \varphi)_{H^{-1}(\Omega)} = \langle f, \varphi \rangle_{H^1_0(\Omega)} = 0.
$$
Hence we have $\|f^s_0\|_{H^{-1}(\Omega)}^2 = (f, f^s_0)_{H^{-1}(\Omega)} = 0$, i.e., $f^s_0 = 0$. The inverse is obvious.

Let $f \in \H_s'$ and let $\bar f \in L^{q'}(\Omega) \hookrightarrow H^{-1}(\Omega)$ be the zero extension of $f$ onto $\Omega$ (see Proposition \ref{P:embed}). Then, for every $\varphi \in E^s_0$, since $\bar f = 0$ a.e.~in $Z(s)$ and $\varphi = 0$ a.e.~in $\Omega \setminus Z(s)$, we deduce that
$$
\langle \bar f, \varphi \rangle_{H^1_0(\Omega)} = \int_{\Omega \setminus Z(s)} f \varphi \, \d x = 0,
$$
that is, $\bar f \in (E^s_0)^\perp$. Hence we deduce from the above that $P_{F^s_0}(\bar f) = 0$.
\end{proof}

Furthermore, we conclude that 
\begin{lemma}\label{Hs'_CONS}
For each $s \geq 0$, the set $\{- \Delta e^s_j/\sqrt{\mu^s_j}\}_{j=1}^\infty$ forms a complete orthonormal system of the associate space $\H_s'$. Hence it holds that
\begin{equation}\label{norm_f-ext}
\|f\|_{\H_s'}^2 = \sum_{j=1}^\infty (\beta^s_j)^2 \mu^s_j, \quad \beta^s_j := \langle \bar f, e^s_j \rangle_{H^1_0(\Omega)} 
\end{equation}
for $f \in \H_s'$. Here $\bar f : \Omega \to \R$ denotes the zero extension of $f$ onto $\Omega$.
\end{lemma}

\begin{proof}
Fix $s \geq 0$. Note that $- \Delta e^s_j = \mu^s_j |v(s)|^{q-2} e^s_j$ vanishes a.e.~in $Z(s)$ and belongs to $\H_s'$ for $j \in \N$ (see Proposition \ref{P:assp}). We see that
\begin{align*}
(- \Delta e^s_i, - \Delta e^s_j)_{\H_s'}
&= \int_{\Omega\setminus Z(s)} (- \Delta e^s_i) (-\Delta e^s_j) |v(s)|^{2-q} \, \d x\\
&= \mu^s_j (e^s_i, e^s_j)_{H^1_0(\Omega)}
= \mu^s_j \delta_{ij}
\end{align*}
for $i,j \in \N$, that is, $\{-\Delta e^s_j/\sqrt{\mu^s_j}\}_{j=1}^\infty$ is an orthonormal system in $\H_s'$. We next prove that $\{-\Delta e^s_j/\sqrt{\mu^s_j}\}_{j=1}^\infty$ is complete in $\H_s'$. Let $f \in \H_s'$ be such that $(f, - \Delta e^s_j)_{\H_s'} = 0$ for all $j \in \N$. Due to Proposition \ref{P:embed}, the zero extension $\bar f$ of $f \in \H_s'$ belongs to $L^{q'}(\Omega) \hookrightarrow H^{-1}(\Omega)$. Noting that
$$
0 = (f, - \Delta e^s_j)_{\H_s'} = \mu^s_j \int_{\Omega \setminus Z(s)} f e^s_j \, \d x = \mu^s_j \langle \bar f, e^s_j \rangle_{H^1_0(\Omega)} \quad \mbox{ for } \ j \in \N
$$
and recalling \eqref{f}, we deduce that
$$
\bar f = P_{F^s_0}(\bar f).
$$
On the other hand, it follows that $P_{F^s_0}(\bar f) = 0$ from Lemma \ref{L:fs0=0} along with $f \in \H_s'$. Thus $f = 0$ in $\H_s'$. Consequently, $\{-\Delta e^s_j/\sqrt{\mu^s_j}\}_{j=1}^\infty$ turns out to be a complete orthonormal system of $\H_s'$.
\end{proof}

\subsection{Taylor expansion of the energy}\label{Ss:Taylor}

This subsection is concerned with a Taylor expansion of the energy functional $J$, which is at least of class $C^2$ in $H^1_0(\Omega)$ but may not be of class $C^3$ (e.g., for $q \in (2,3)$).
\begin{lemma}[Taylor expansion of the energy]\label{L:Taylor}
For each $s \geq 0$, it holds that
\begin{align}
 J(v(s)) - J(\phi)
&= \frac 12 \langle \Lphi (v(s)-\phi), v(s)-\phi \rangle_{H^1_0(\Omega)} + E(s)
\label{Tay1}
\end{align}
and
\begin{equation}\label{Tay2}
J'(v(s)) = \Lphi (v(s)-\phi) + e(s),
\end{equation}
where $s \mapsto E(s) \in \R$ and $s \mapsto e(s) \in H^{-1}(\Omega)$ denote generic functions satisfying
\begin{equation}\label{e}
E(s) \leq C \|v(s)-\phi\|_{L^q(\Omega)}^{2+\rho} \ \mbox{ and } \ \|e(s)\|_{H^{-1}(\Omega)} \leq C\|v(s)-\phi\|_{L^q(\Omega)}^{1+\rho}
\end{equation}
with $\rho = \min\{1,q-2\}>0$. Here the constant $C$ depends only on $q$, $C_q$ and $c(v)$ given by \eqref{vLq}.
\end{lemma}

\begin{proof}
Fix $s \geq 0$. In case $2 < q < 3$, by direct computation, we infer that
\begin{align*}
e(s) &:= J'(v(s)) - \Lphi (v(s)-\phi)\\
 &= - \lambda_q \left[ |v(s)|^{q-2}v(s) - |\phi|^{q-2}\phi - (q-1)|\phi|^{q-2}(v(s)-\phi) \right]\\
 &= - \lambda_q (q-1) \left[ |(1-\theta)v(s)+\theta\phi|^{q-2} - |\phi|^{q-2} \right] (v(s)-\phi),
\end{align*}
where $\theta \in (0,1)$ may depend on $x$ and $s$. Hence we observe that, for $\varphi \in H^1_0(\Omega)$,
\begin{align*}
\MoveEqLeft{
| \langle e(s), \varphi \rangle_{H^1_0(\Omega)} |
}\\
&\leq \lambda_q (q-1) \int_\Omega \left| |(1-\theta)v(s)+\theta \phi|^{q-2} - |\phi|^{q-2} \right| |v(s)-\phi| |\varphi|\, \d x\\
&\leq \lambda_q (q-1) \int_\Omega |1-\theta|^{q-2} |v(s)-\phi|^{q-1} |\varphi|\, \d x\\
&\leq \lambda_q (q-1) \int_\Omega |v(s)-\phi|^{q-1} |\varphi|\, \d x,
\end{align*}
which along with the arbitrariness of $\varphi \in H^1_0(\Omega)$ implies
$$
\|e(s)\|_{H^{-1}(\Omega)} \leq \lambda_q (q-1) C_q \|v(s)-\phi\|_{L^q(\Omega)}^{q-1}.
$$

Moreover, it follows from \eqref{Tay1} that
\begin{align*}
E(s) :=&\ J(v(s)) - J(\phi) - \frac 12 \langle \Lphi (v(s)-\phi), v(s)-\phi \rangle_{H^1_0(\Omega)}\\
=&\ \frac 1 2 \|\nabla v(s)\|_{L^2(\Omega)}^2 - \frac 1 2 \|\nabla \phi\|_{L^2(\Omega)}^2 - \frac{\lambda_q}{q} \|v(s)\|_{L^q(\Omega)}^q + \frac{\lambda_q}{q} \|\phi\|_{L^q(\Omega)}^q \\
&\quad - \frac 1 2 \|\nabla (v(s) - \phi)\|_{L^2(\Omega)}^2 + \frac{\lambda_q}2 (q-1)\int_\Omega |\phi|^{q-2} (v(s)-\phi)^2 \, \d x\\
=&\ \int_\Omega \nabla (v(s)-\phi) \cdot \nabla \phi \, \d x- \frac{\lambda_q}{q} \|v(s)\|_{L^q(\Omega)}^q + \frac{\lambda_q}{q} \|\phi\|_{L^q(\Omega)}^q \\
&\quad + \frac{\lambda_q}2 (q-1)\int_\Omega |\phi|^{q-2} (v(s)-\phi)^2 \, \d x\\
=&\ \lambda_q \int_\Omega |\phi|^{q-2}\phi (v(s)-\phi) \, \d x- \frac{\lambda_q}{q} \|v(s)\|_{L^q(\Omega)}^q + \frac{\lambda_q}{q} \|\phi\|_{L^q(\Omega)}^q \\
&\quad + \frac{\lambda_q}2 (q-1)\int_\Omega |\phi|^{q-2} (v(s)-\phi)^2 \, \d x\\
=&\ \frac{\lambda_q}2(q-1) \int_\Omega \left( -|(1-\theta)v(s)+\theta \phi|^{q-2} + |\phi|^{q-2} \right) (v(s)-\phi)^2 \, \d x
\end{align*}
for some constant $\theta \in (0,1)$ which may depend on $x$ and $s$. Hence one can similarly verify that
\begin{align*}
|E(s)|
 \leq \frac{\lambda_q}2 (q-1) \|v(s)-\phi\|_{L^q(\Omega)}^q
 \leq \frac{\lambda_q}2 (q-1) C_q^q \|v(s)-\phi\|_{H^1_0(\Omega)}^q.
\end{align*}
Thus \eqref{e} with $\rho = q-2$ follows.

In case $q \geq 3$, as in the proof of Lemma \ref{L:ev-pert}, we can also derive \eqref{Tay1} and \eqref{Tay2} along with \eqref{e} and $\rho = 1$. Then the constant $C$ may further depend on $c(v)$ as well.
\end{proof}

\subsection{Quantitative gradient inequality}\label{Ss:GI}

The following lemma provides a quantitative gradient inequality for $J(\cdot)$ and will play a crucial role in the proof of Theorem \ref{T:sc-conv}:

\begin{lemma}[Quantitative gradient inequality]\label{L:GI}
There exist constants $s_1 \geq 0$ and $C > 0$ such that
\begin{align}
 J(v(s)) - J(\phi)
&\leq \left( \frac 1 {2 \nu_k^s} + C \|v(s)-\phi\|_{H^1_0(\Omega)}^\rho \right) \| J'(v(s)) \|_{\H_s'}^2 \label{GI}
\end{align}
for all $s \geq s_1$. Here $\nu^s_k$ denotes the smallest positive eigenvalue of $\Ls$ {\rm (}see Remark {\rm \ref{R:Ls-inv}}{\rm )} and $\rho = \min\{1,q-2\} \in (0,1]$. Moreover, the constant $C$ depends only on $q$, $C_q$, $c(v)$ given in \eqref{vLq} and $\|\Lphi^{-1}\|_{\mathscr{L}(H^{-1}(\Omega),H^1_0(\Omega))}$ {\rm (}see Lemma {\rm\ref{L:Lsinv-bdd}}{\rm )}
\end{lemma}

\begin{proof}
Fix $s \geq s_0$ large enough in view of Remark \ref{R:Ls-inv} (see also Lemma \ref{L:Lsinv-bdd}). Let $f \in \H_s'$ be fixed and let $\bar f \in H^{-1}(\Omega)$ be the zero extension of $f$ onto $\Omega$. Then $\bar f$ can be expanded as in \eqref{f} with $f^s_0 = 0$ and $\beta^s_j := \langle \bar f, e^s_j \rangle_{H^1_0(\Omega)}$ for $j \in \N$ (see Lemma \ref{L:fs0=0}). Moreover, by virtue of \eqref{fLinvf_0} and Lemma \ref{Hs'_CONS}, we have
\begin{align}
\langle \bar f, \Ls^{-1} \bar f \rangle_{H^1_0(\Omega)}
&= \sum_{j=1}^\infty (\beta^s_j)^2 \frac{\mu_j^s}{\mu_j^s - \lambda_q (q-1)}\nonumber\\
&\leq \frac{1}{\mu_k^s - \lambda_q (q-1)} \sum_{j=k}^\infty (\beta^s_j)^2 \mu_j^s\nonumber\\
&\leq \frac{1}{\mu_k^s - \lambda_q (q-1)} \|f\|_{\H_s'}^2,\label{fLinvf}
\end{align}
where $\Ls = - \Delta - \lambda_q (q-1) |v(s)|^{q-2}$. 

We observe that
\begin{align}
J'(v(s)) &\stackrel{\eqref{Tay2}}{=} \Lphi (v(s)-\phi) + e(s)\nonumber\\
&= \Ls (v(s)-\phi) - \lambda_q (q-1) \left( |\phi|^{q-2} - |v(s)|^{q-2} \right) \left( v(s)-\phi \right)\nonumber\\
&\quad + e(s),\label{Tay1.5}
\end{align}
which implies that
\begin{align}
v(s)-\phi &= \Ls^{-1} \circ J'(v(s)) \nonumber \\
&\quad + \lambda_q (q-1) \Ls^{-1} \left[ \left( |\phi|^{q-2} - |v(s)|^{q-2} \right) \left( v(s)-\phi \right) \right] \nonumber\\
&\quad - \Ls^{-1} \left(e(s)\right).\label{v-phi}
\end{align}
We can derive from \eqref{power-diff} that
\begin{equation}\label{err-est}
\left\| \left( |\phi|^{q-2} - |v|^{q-2} \right) \left( v(s)-\phi \right) \right\|_{H^{-1}(\Omega)}
\leq c \|v(s)-\phi\|_{H^1_0(\Omega)}^{\rho+1},
\end{equation}
where $\rho = \min\{1,q-2\} \in (0,1]$ and $c$ depends only on $q$, $C_q$ and $c(v)$ given by \eqref{vLq}. Thus by \eqref{Tay2} of Lemma \ref{L:Taylor} and \eqref{Tay1.5}--\eqref{err-est}, we obtain
\begin{align*}
\MoveEqLeft{
\langle \Lphi (v(s)-\phi), v(s)-\phi \rangle_{H^1_0(\Omega)}
}\\
&= \langle J'(v(s)), v(s)-\phi \rangle_{H^1_0(\Omega)} - \langle e(s), v(s)-\phi \rangle_{H^1_0(\Omega)}\\
&= \langle J'(v(s)), \Ls^{-1} \circ J'(v(s)) \rangle_{H^1_0(\Omega)} \nonumber\\
&\quad + \lambda_q(q-1) \left\langle J'(v(s)), \Ls^{-1} \left[ \left( |\phi|^{q-2} - |v(s)|^{q-2} \right) \left( v(s)-\phi \right) \right] \right\rangle_{H^1_0(\Omega)}\nonumber\\
&\quad - \langle J'(v(s)), \Ls^{-1} (e(s)) \rangle_{H^1_0(\Omega)} - \langle e(s), v(s)-\phi \rangle_{H^1_0(\Omega)}\\
&\leq \langle J'(v(s)), \Ls^{-1} \circ J'(v(s)) \rangle_{H^1_0(\Omega)} \\
&\quad + C \left( \|v(s)-\phi\|_{H^1_0(\Omega)}^{\rho+2} + \|J'(v(s))\|_{H^{-1}(\Omega)} \|v(s)-\phi\|_{H^1_0(\Omega)}^{\rho+1}\right),
\end{align*}
where $C$ is a constant depending only on $q$, $C_q$, $c(v)$ given by \eqref{vLq} and $\|\Lphi^{-1}\|_{\mathscr{L}(H^{-1}(\Omega),H^1_0(\Omega))}$ (see Lemma \ref{L:Lsinv-bdd}).

Therefore combining this and \eqref{Tay1} along with \eqref{e}, we deduce that
\begin{align}
\lefteqn{
 J(v(s)) - J(\phi)
\leq \frac 12 \langle J'(v(s)), \Ls^{-1} \circ J'(v(s)) \rangle_{H^1_0(\Omega)}
}\nonumber\\
&\quad + C \left( \|v(s)-\phi\|_{H^1_0(\Omega)}^{\rho+2} + \|J'(v(s))\|_{H^{-1}(\Omega)} \|v(s)-\phi\|_{H^1_0(\Omega)}^{\rho+1}\right).
\label{GI-1}
\end{align}
Since $J'(v(s)) \in \H_s'$ (see \S \ref{Ss:EI}), it follows from \eqref{fLinvf} that
\begin{align}
\lefteqn{
J(v(s)) - J(\phi) \leq \frac 1 {2 \nu_k^s} \| J'(v(s)) \|_{\H_s'}^2
} \nonumber\\
&+ C \left( \|v(s)-\phi\|_{H^1_0(\Omega)}^{\rho+2} + \|J'(v(s))\|_{H^{-1}(\Omega)} \|v(s)-\phi\|_{H^1_0(\Omega)}^{\rho+1}\right).\label{GI0}
\end{align}
Moreover, we find from \eqref{Tay2} along with \eqref{e} again that
\begin{align*}
\|v(s)-\phi\|_{H^1_0(\Omega)}
& \leq \|\Lphi^{-1}\|_{\mathscr{L}(H^{-1}(\Omega);H^1_0(\Omega))} \|J'(v(s))\|_{H^{-1}(\Omega)}\\
&\quad + \|\Lphi^{-1}\|_{\mathscr{L}(H^{-1}(\omega),H^1_0(\Omega))} \|e(s)\|_{H^{-1}(\Omega)}\\
&\leq \|\Lphi^{-1}\|_{\mathscr{L}(H^{-1}(\omega),H^1_0(\Omega))} \|J'(v(s))\|_{H^{-1}(\Omega)}\\
&\quad + C \|\Lphi^{-1}\|_{\mathscr{L}(H^{-1}(\omega),H^1_0(\Omega))} \|v(s)-\phi\|_{H^1_0(\Omega)}^{\rho+1}.
\end{align*}
Since $\|v(s)-\phi\|_{H^1_0(\Omega)}$ is small enough, e.g., smaller than the constant $(2 C \|\Lphi^{-1}\|_{\mathscr{L}(H^{-1}(\omega),H^1_0(\Omega))})^{-1/\rho}$, for $s \geq 0$ large enough (see \eqref{H10-conv}), we get
\begin{align}\label{GIp}
\|v(s)-\phi\|_{H^1_0(\Omega)} 
\leq C \|J'(v(s))\|_{H^{-1}(\Omega)}
\leq C \|J'(v(s))\|_{\H_s'}
\end{align}
for $s > 0$ large enough. Here we used the fact that $J'(v(s))$ vanishes on $Z(s)$ and $J'(v(s))|_{\Omega \setminus Z(s)}$ lies on $\H_s'$ (see \eqref{time-deri-decomp}); hence $J'(v(s))$ coincides with the zero extension of $J'(v(s))|_{\Omega \setminus Z(s)}$ onto $\Omega$ (see also Proposition \ref{P:embed}). Combining this with \eqref{GI0}, we can take $s_1 \geq s_0$ large enough such that
\begin{align*}
 J(v(s)) - J(\phi)
&\leq \left( \frac 1 {2 \nu_k^s} + C \|v(s)-\phi\|_{H^1_0(\Omega)}^\rho \right) \| J'(v(s)) \|_{\H_s'}^2
\end{align*}
for $s \geq s_1$. Here the constant $C$ eventually depends only on $q$, $C_q$, $c(v)$ given by \eqref{vLq} and $\|\Lphi^{-1}\|_{\mathscr{L}(H^{-1}(\Omega),H^1_0(\Omega))}$ (see Lemma \ref{L:Lsinv-bdd}). This completes the proof.
\end{proof}

In particular, if $\phi$ is a nondegenerate least-energy solution to \eqref{eq:1.10}, \eqref{eq:1.11}, we can also take $s_0 = s_1 = 0$ whenever $v_0$ lies on $\mathcal{X}$ and is close enough to $\phi$ in $H^1_0(\Omega)$. Indeed, thanks to the stability result in~\cite[Theorem 2]{AK13}, we can then observe that $\|v(s)-\phi\|_{H^1_0(\Omega)}$ is small enough for any $s \geq 0$.

\subsection{Sharp rate of convergence}\label{Ss:SRC}

Combining Lemma \ref{L:GI} with \eqref{eneq2}, we infer that
\begin{align}
\MoveEqLeft
\frac 1 {q-1} \left( \frac 1 {2 \nu_k^s} + C \|v(s)-\phi\|_{H^1_0(\Omega)}^\rho \right)^{-1} \left[ J(v(s)) - J(\phi) \right] 
\nonumber\\
&\leq - \dfrac{\d}{\d s} J(v(s))\label{conclu*}
\end{align}
for $s \geq s_1$. Recalling \eqref{ev-pert}, we can take a constant $C$ such that
$$
\frac1{\nu_k^s} \leq \frac1{\nu_k} + C \|v(s)-\phi\|_{H^1_0(\Omega)}^{\rho}
$$
for $s \geq s_1$ large enough (here and henceforth, $s_1$ is replaced by such a large number). Indeed, by virtue of \eqref{ev-pert} and the mean-value theorem, we see that
\begin{align*}
 \frac 1 {\nu^s_k} &= \frac{1/\mu^s_k}{1 - \lambda_q(q-1)/\mu^s_k}\\
&\leq \frac{1/\mu_k + C \|v(s)-\phi\|_{H^1_0(\Omega)}^\rho}{1 - \lambda_q(q-1)/\mu_k - \lambda_q (q-1) C \|v(s)-\phi\|_{H^1_0(\Omega)}^\rho}\\
&\leq \frac 1 {\nu_k} + C \left( \frac{\mu_k^2}{\nu_k^2}+1\right) \|v(s)-\phi\|_{H^1_0(\Omega)}^\rho
\end{align*}
for $s \geq s_1$ large enough so that the denominator of the second line above is positive (see \eqref{H10-conv}). 

Set $H(s) := J(v(s)) - J(\phi)$ for $s \geq 0$. It then follows that
\begin{align}\label{H-ineq}
\dfrac{\d H}{\d s}(s) + \frac{2\nu_k}{q-1} H(s)
\leq C \|v(s)-\phi\|_{H^1_0(\Omega)}^\rho H(s)
\quad \mbox{ for } \ s \geq s_1,
\end{align}
where $C$ depends only on $q$, $C_q$, $c(v)$ given by \eqref{vLq}, $\mu_k$ and \\$\|\Lphi^{-1}\|_{\mathscr{L}(H^{-1}(\Omega),H^1_0(\Omega))}$ (see Lemma \ref{L:Lsinv-bdd}). Thus due to \eqref{H10-conv} for any $\lambda \in (0,\lambda_0)$ one can take a constant $C_\lambda > 0$ such that
$$
0 \leq J(v(s)) - J(\phi) \leq C_\lambda \e^{-\lambda s} \quad \mbox{ for } \ s \geq 0.
$$
Here we also used the fact that $J(v(s)) \leq J(v_0)$ for $s \geq 0$. 

On the other hand, let $\lambda \in (0,\lambda_0)$ be fixed. In particular, if $\phi$ is a nondegenerate least-energy solution to \eqref{eq:1.10}, \eqref{eq:1.11}, we can then assure that $\sup_{s \geq 0}\|v(s)-\phi\|_{H^1_0(\Omega)}$ is small enough and take $s_0 = s_1 = 0$, whenever $v_0 \in \mathcal X$  and $\|v_0 - \phi\|_{H^1_0(\Omega)} \ll 1$ (see~\cite[Theorem 2]{AK13}); therefore we can obtain
$$
0 \leq J(v(s)) - J(\phi) \leq C_\lambda \left(J(v_0) - J(\phi)\right) \e^{-\lambda s} \quad \mbox{ for } \ s \geq 0.
$$
Here we stress that $C_\lambda$ can be chosen as a constant independent of $v_0$ and $s$ (when $\phi$ and $v_0$ fulfill the assumptions mentioned just above).

Now, we prove the following lemma:
\begin{lemma}\label{L:JtoH10}
Assume that
\begin{equation}\label{p:ass}
0 \leq J(v(s)) - J(\phi) \leq c \e^{-\lambda s} \quad \mbox{ for } s \geq 0
\end{equation}
for some constants $\lambda > 0$ and $c > 0$. Then there exists a constant $C > 0$ such that
\begin{equation}\label{cl1}
\|v(s)-\phi\|_{H^1_0(\Omega)}^2 \leq C \e^{-\lambda s} \quad \mbox{ for } s \geq 0.
\end{equation}
In particular, let $\phi$ be a nondegenerate \emph{least-energy} solution to \eqref{eq:1.10}, \eqref{eq:1.11}. Then there exist constants $\delta > 0$ and $M \geq 0$ such that
\begin{equation}\label{cl2}
\|v(s)-\phi\|_{H^1_0(\Omega)}^2 \leq c M \e^{-\lambda s} \quad \mbox{ for } s \geq 0, 
\end{equation}
where $v$ is the energy solution to \eqref{eq:1.6}--\eqref{eq:1.8} with the initial datum $v_0$, provided that $v_0 \in \mathcal{X}$, $\|v_0 - \phi\|_{H^1_0(\Omega)} < \delta$ and \eqref{p:ass} holds.
\end{lemma}

To prove this lemma, recall an entropy functional $K : H^1_0(\Omega) \to \R$ defined by
$$
K(w) = \frac{1}{q'} \|w\|_{L^q(\Omega)}^q - \frac{\lambda_q}2\left\|\nabla (-\Delta)^{-1}\left(|w|^{q-2}w\right)\right\|_{L^2(\Omega)}^2 \ \mbox{ for } \  w \in H^1_0(\Omega),
$$
which is another Lyapunov functional, that is, $s \mapsto K(v(s))$ is nonincreasing for every energy solution $v = v(x,s)$ to \eqref{eq:1.6}--\eqref{eq:1.8} (see~\cite[p.567]{AK13}). The following lemma provides a coercive estimate for the functional $G : H^1_0(\Omega) \to \R$ given by
\begin{align*}
G(w) &:= J(w) - \lambda_q K(w)\\
&= \frac 12 \|\nabla w\|_{L^2(\Omega)}^2 - \lambda_q \|w\|_{L^q(\Omega)}^q + \frac{\lambda_q^2}2 \left\|\nabla (-\Delta)^{-1}\left(|w|^{q-2}w\right)\right\|_{L^2(\Omega)}^2 
\end{align*}
for $w \in H^1_0(\Omega)$. One can directly check that $G(\phi)=0$ if $J'(\phi)=0$, and $G(w)$ will play a crucial role to prove Lemma \ref{L:JtoH10}. Moreover, the following lemma may also be of independent interest.

\begin{lemma}[Coercivity estimate for $G$ near $\phi$]\label{L:coer}
For the functional $G$ defined above, it holds that
\begin{align*}
G(w)= \frac 12 \|J'(w)\|_{H^{-1}(\Omega)}^2
\end{align*}
for all $w \in H^1_0(\Omega)$. As a corollary, $G(w)=0$ if and only if $J'(w)=0$. In addition, if $\phi$ is a weak solution to \eqref{eq:1.10}, \eqref{eq:1.11} {\rm (}that is, $J'(\phi) = 0${\rm )}, and if $\phi$ is nondegenerate, then for any $\vep\in (0,1)$ there exists a constant $\delta_\epsilon>0$ such that
\begin{align}\label{G-coer}
G(w) - G(\phi) \geq \frac{1-\vep}2 \|\Lphi^{-1}\|_{\mathscr{L}(H^{-1}(\Omega), H^1_0(\Omega))}^{-2} \|\nabla w - \nabla \phi\|_{L^2(\Omega)}^2,
\end{align}
provided that $w \in H^1_0(\Omega)$ and $\|w - \phi\|_{L^q(\Omega)} < \delta_\vep$. In particular, if $K(w)-K(\phi) \geq -c \|\nabla w - \nabla \phi\|_{L^2(\Omega)}^2$ for some constant $c$ satisfying $0 < c < (2\lambda_q)^{-1} \|\Lphi^{-1}\|_{\mathscr{L}(H^{-1}(\Omega), H^1_0(\Omega))}^{-2}$, then \eqref{G-coer} further gives a strict coercive estimate for $J(w)-J(\phi)$.
\end{lemma}

\begin{proof}
By direct computation, we have, for $w\in H^1_0(\Omega)$,
\begin{align*}
\lefteqn{
G(w) 
}\\
&= \frac 12 \|\nabla w \|_{L^2(\Omega)}^2 + \frac{\lambda_q^2}2 \left\| \nabla (-\Delta)^{-1} (|w|^{q-2}w ) \right\|_{L^2(\Omega)}^2 - \lambda_q \int_\Omega |w|^q  \, \d x\\
&= \frac 12 \|-\Delta w \|_{H^{-1}(\Omega)}^2 + \frac{\lambda_q^2}2 \left\| |w|^{q-2}w\right\|_{H^{-1}(\Omega)}^2 - \lambda_q \langle -\Delta w, |w|^{q-2}w\rangle_{H^{-1}(\Omega)} \\
&= \frac 12 \left\| -\Delta w - \lambda_q  |w|^{q-2}w \right\|_{H^{-1}(\Omega)}^2\\
&= \frac 12 \left\| J'(w)\right\|_{H^{-1}(\Omega)}^2.
\end{align*}
Next we write $J'(w) = \Lphi (w -\phi) -\lambda_q \mathcal{R} (w,\phi)$ by using $J'(\phi)=0$, where the residual term $\mathcal{R}(w,\phi) \in H^{-1}(\Omega)$ is given by
\begin{equation}\label{R}
\mathcal{R}(w,\phi) := |w|^{q-2}w - |\phi|^{q-2}\phi - (q-1) |\phi|^{q-2} (w-\phi)
\end{equation}
and fulfills that
\begin{align}\label{R-est}
\lefteqn{
\|\mathcal{R}(w,\phi)\|_{L^{q'}(\Omega)}
}\nonumber\\
&\leq \begin{cases}
      (q-1)\|w-\phi\|_{L^q(\Omega)}^{q-1} &\mbox{if } \ q \in (2,3),\\
      \frac{(q-1)(q-2)}2 \left( \|w\|_{L^q(\Omega)}^{q-3} + \|\phi\|_{L^q(\Omega)}^{q-3}\right) \|w-\phi\|_{L^q(\Omega)}^2 &\mbox{if } \ q \geq 3.
     \end{cases}
\end{align}
Then we observe that
\begin{align*}
\lefteqn{
\|\Lphi(w-\phi) - \lambda_q \mathcal{R}(w,\phi)\|_{H^{-1}(\Omega)}
}\\
&\geq \|\Lphi(w-\phi)\|_{H^{-1}(\Omega)} - \lambda_q \|\mathcal{R}(w,\phi)\|_{H^{-1}(\Omega)}\\
&\geq \|\Lphi^{-1}\|_{\mathscr{L}(H^{-1}(\Omega),H^1_0(\Omega))}^{-1} \| w-\phi\|_{H^1_0(\Omega)} - \lambda_q C_q \|\mathcal{R}(w,\phi)\|_{L^{q'}(\Omega)}.
\end{align*}
Hence for any $\vep \in (0,1)$ one can take $\delta_\vep > 0$ small enough that
\begin{align*}
\lefteqn{
\|\Lphi(w-\phi) - \lambda_q \mathcal{R}(w,\phi)\|_{H^{-1}(\Omega)}
}\\
&\geq \sqrt{1-\vep} \|\Lphi^{-1}\|_{\mathscr{L}(H^{-1}(\Omega),H^1_0(\Omega))}^{-1} \|\nabla w-\nabla \phi\|_{L^2(\Omega)},
\end{align*}
provided that $\|w - \phi\|_{L^q(\Omega)} < \delta_\vep$. Thus the latter assertion follows. This completes the proof. 
\end{proof}

Now, we are ready to prove Lemma \ref{L:JtoH10}.

\begin{proof}[Proof of Lemma {\rm \ref{L:JtoH10}}]
Setting $\vep = 1/2$ and recalling \eqref{H10-conv}, one can take $s_* > 0$ large enough that
$$
\sup_{s \geq s_*} \|v(s)-\phi\|_{L^q(\Omega)} < \delta_\vep,
$$
where $\delta_\vep > 0$ is the constant appeared in Lemma \ref{L:coer}. Hence it follows from Lemma \ref{L:coer} that
\begin{equation}\label{Gv-coer}
 G(v(s)) - G(\phi) \geq \frac 1 4 \|\Lphi^{-1}\|_{\mathscr{L}(H^{-1}(\Omega), H^1_0(\Omega))}^{-2} \|\nabla v(s) - \nabla \phi\|_{L^2(\Omega)}^2
\end{equation}
for $s \geq s_*$. Moreover, we also recall that
$$
K(v(s)) - K(\phi) \geq 0 \quad \mbox{ for } \ s \geq 0,
$$
which along with \eqref{Gv-coer} implies
\begin{align*}
\|\nabla v(s) - \nabla \phi\|_{L^2(\Omega)}^2
&\leq 4 \|\Lphi^{-1}\|_{\mathscr{L}(H^{-1}(\Omega), H^1_0(\Omega))}^2 \left( J(v(s)) - J(\phi) \right) \\
&\leq 4 \|\Lphi^{-1}\|_{\mathscr{L}(H^{-1}(\Omega), H^1_0(\Omega))}^2 c \e^{-\lambda s}
\quad \mbox{ for } \ s \geq s_*. 
\end{align*}
Since $v(s)$ is bounded in $H^1_0(\Omega)$ for any $s \geq 0$, \eqref{cl1} follows.

In particular, if $\phi$ is a nondegenerate least-energy solution to \eqref{eq:1.10}, \eqref{eq:1.11} , thanks to~\cite[Theorem 2]{AK13}, for any $\vep > 0$, one can take $\delta > 0$ such that $\sup_{s \geq 0}\|v(s)-\phi\|_{H^1_0(\Omega)} < \vep$, where $v$ is the energy solution to \eqref{eq:1.6}--\eqref{eq:1.8} with the initial datum $v_0$, whenever $v_0 \in \mathcal{X}$ and $\|v_0 - \phi\|_{H^1_0(\Omega)} < \delta$. Hence we can take $s_* = 0$. Thus \eqref{cl2} follows. This completes the proof.
\end{proof} 

\begin{remark}[An alternative proof]
We can also prove Lemma \ref{L:JtoH10} as in~\cite[Lemma 4.1]{A21} with slight modifications due to the time-dependence of $\H_s$ and $\H_s'$.
\end{remark}

Now, we are in a position to prove main results.

\begin{proof}[Proof of Theorem {\rm \ref{T:sc-conv}}]
Thanks to Lemma \ref{L:JtoH10} we have
$$
\|v(s)-\phi\|_{H^1_0(\Omega)} \leq c \e^{-\lambda s} \quad \mbox{ for } \ s \geq 0
$$
for some constant $c > 0$. Thus it follows from \eqref{H-ineq} that
$$
\dfrac{\d H}{\d s}(s) + \frac{2\nu_k}{q-1} H(s)
\leq C c^\rho \e^{-\lambda \rho s} H(s) \quad \mbox{ for } \ s \geq s_1.
$$
Hence there exists a constant $M > 0$ such that
\begin{equation}\label{J-ec}
0 \leq H(s) \leq M H(s_1) \e^{-\lambda_0 (s - s_1)} \quad \mbox{ for } \ s \geq s_1,
\end{equation}
where $\lambda_0 = 2\nu_k/(q-1)$. Thus \eqref{J-conv} follows, since $v(s)$ is bounded in $H^1_0(\Omega)$ for $s \geq 0$. Furthermore, the assertion \eqref{H10-ec} follows from Lemma \ref{L:JtoH10}. This completes the proof of Theorem \ref{T:sc-conv}.
\end{proof}

\begin{proof}[Proof of Corollary {\rm \ref{C:stbl}}]
Suppose that $\phi$ is a nondegenerate least-energy solution to \eqref{eq:1.10}, \eqref{eq:1.11}. Thanks to~\cite[Theorem 2]{AK13}, for any $\vep > 0$ one can take $\delta > 0$ such that $\sup_{s \geq 0} \|v(s)-\phi\|_{H^1_0(\Omega)} < \vep$, where $v$ denotes the energy solution to \eqref{eq:1.6}--\eqref{eq:1.8} with an initial datum $v_0$, whenever $v_0 \in \mathcal{X}$ and $\|v_0-\phi\|_{H^1_0(\Omega)} < \delta$ (in particular, $c(v)$ given in \eqref{vLq} is uniformly bounded for the choice of $v_0 \in \mathcal{X}$ in the $\delta$-neighbourhood of $\phi$). Therefore we can take $s_1 = 0$, and consequently, there exists a constant $M \geq 0$ (independent of $v$ and $s$) such that
$$
0 \leq H(s) \leq M H(0) \e^{-\lambda_0s} \quad \mbox{ for } \ s \geq 0,
$$
which along with Lemma \ref{L:JtoH10} implies the desired conclusion of Corollary \ref{C:stbl}. 
\end{proof}

\section{An alternative proof with an $\vep$-regularization}\label{S:alt}

In the last section, in order to prove Theorem \ref{T:sc-conv}, we derived \eqref{conclu*} based on the spectral decomposition of $J'(v(s))$ in the associate space $\H_s'$ of the weighted $L^2$-space $\H_s$ (see \S \ref{Ss:wL2} and \S \ref{Ss:dual-decomp}). To this end, we paid a careful attention to the singularity of the weight function $|v(s)|^{2-q}$ of $\H_s'$ on the set $Z(s)$ of zeros of $v(s)$. In this section, instead of using the associate space $\H_s'$, we shall introduce an $\vep$-approximation for the singular weight and derive \eqref{conclu*} in another fashion. 

\subsection{A modified energy inequality with an $\vep$-regularization}

Let us recall the relation used in the last section,
\begin{align}\label{ut-rel-old}
\left\| \partial_s (|v|^{q-2}v)(s) \right\|_{\H_s'}^2
&= \frac{4(q-1)^2}{q^2} \left\| \partial_s (|v|^{(q-2)/2}v)(s)\right\|_{L^2(\Omega)}^2,
\end{align}
the left-hand side of which is now approximated as
$$
\left\langle
\partial_s (|v|^{q-2}v)(s),
(-\vep \Delta + |v(s)|^{q-2})^{-1} \partial_s (|v|^{q-2}v)(s)
\right\rangle_{H^1_0(\Omega)}
$$
for $\vep > 0$ (then $\H_s'$ will no longer appear in what follows). Here we note that 
$$
-\vep \Delta + |v(s)|^{q-2}
= (-\Delta) \circ \left( \vep I + A^s \right),
$$
which turns out to be a bijective and bounded operator from $H^1_0(\Omega)$ into $H^{-1}(\Omega)$ (see \S \ref{Ss:As}). It also follows that
\begin{align}
(-\vep \Delta + |v(s)|^{q-2})^{-1}
&= \left[ (-\Delta) \circ (\vep I + A^s) \right]^{-1}\nonumber\\
&= \left( \vep I + A^s \right)^{-1} \circ (-\Delta)^{-1}.\label{alt0.5}
\end{align}

We claim that
\begin{align}
\MoveEqLeft{
\left|
\left\langle
\partial_s (|v|^{q-2}v)(s),
(-\vep \Delta + |v(s)|^{q-2})^{-1} \partial_s (|v|^{q-2}v)(s)
\right\rangle_{H^1_0(\Omega)}
\right|
}\nonumber\\
&\leq \frac{4(q-1)^2}{q^2} \|\partial_s (|v|^{(q-2)/2}v)(s)\|_{L^2(\Omega)}^2\label{alt1.1}
\end{align}
for $\vep > 0$\/; it will be used below instead of \eqref{ut-rel-old}. Indeed, set $f = (-\vep \Delta + |v(s)|^{q-2})^{-1} \partial_s (|v|^{q-2}v)(s) \in H^1_0(\Omega)$. Then we see that
\begin{equation*}
- \vep \Delta f + |v(s)|^{q-2}f = \partial_s (|v|^{q-2}v)(s) \ \mbox{ in } H^{-1}(\Omega).
\end{equation*}
Hence testing it by $f$, we have
\begin{align*}
\MoveEqLeft{
\vep \|\nabla f\|_{L^2(\Omega)}^2 + \int_\Omega |v(s)|^{q-2}|f|^2 \, \d x
}\\
&= \left\langle \partial_s (|v|^{q-2}v)(s), f \right\rangle_{H^1_0(\Omega)}\\
&\hspace{-1.6mm}\stackrel{\eqref{time-deri-decomp}}= \frac{2(q-1)}q \left\langle |v(s)|^{(q-2)/2} \partial_s (|v|^{(q-2)/2}v)(s), f \right\rangle_{H^1_0(\Omega)}\\
&= \frac{2(q-1)}q \left( \partial_s (|v|^{(q-2)/2}v)(s), |v(s)|^{(q-2)/2} f \right)_{L^2(\Omega)}\\
&\leq \frac{2(q-1)}q \|\partial_s (|v|^{(q-2)/2}v)(s)\|_{L^2(\Omega)} \||v(s)|^{(q-2)/2} f\|_{L^2(\Omega)},
\end{align*}
whence it follows that
\begin{equation}\label{utestA}
\||v(s)|^{(q-2)/2} f\|_{L^2(\Omega)} \leq \frac{2(q-1)}q \|\partial_s (|v|^{(q-2)/2}v)(s)\|_{L^2(\Omega)}.
\end{equation}
Thus we obtain
\begin{align*}
\MoveEqLeft{
\left|
\left\langle
\partial_s (|v|^{q-2}v)(s),
(-\vep \Delta + |v(s)|^{q-2})^{-1} \partial_s (|v|^{q-2}v)(s)
\right\rangle_{H^1_0(\Omega)}
\right|
}\nonumber\\
&\stackrel{\eqref{time-deri-decomp}}= \frac{2(q-1)}q \left|
\left(
\partial_s (|v|^{(q-2)/2}v)(s), |v(s)|^{(q-2)/2} f 
\right)_{L^2(\Omega)}
\right|\nonumber\\
&\stackrel{\eqref{utestA}}\leq \frac{4(q-1)^2}{q^2} \left\|
\partial_s (|v|^{(q-2)/2}v)(s)
\right\|_{L^2(\Omega)}^2.
\end{align*}
Thus, we have proved \eqref{alt1.1}. On the other hand, using \eqref{alt0.5}, we observe that
\begin{align}
\MoveEqLeft{
\left\langle
\partial_s (|v|^{q-2}v)(s),
(-\vep \Delta + |v(s)|^{q-2})^{-1} \partial_s (|v|^{q-2}v)(s)
\right\rangle_{H^1_0(\Omega)}
}\nonumber\\
&= \left( (-\Delta)^{-1} \partial_s (|v|^{q-2}v)(s), 
(\vep I + A^s)^{-1} \circ (-\Delta)^{-1} \partial_s (|v|^{q-2}v)(s)\right)_{H^1_0(\Omega)}\nonumber\\
&= \left\| (\vep I + A^s)^{-1/2} \circ (-\Delta)^{-1} \partial_s (|v|^{q-2}v)(s) \right\|_{H^1_0(\Omega)}^2\nonumber\\
&= \left\| (\vep I + A^s)^{-1/2} \circ (-\Delta)^{-1} \circ J'(v(s)) \right\|_{H^1_0(\Omega)}^2,\label{alt2}
\end{align}
which may correspond to the $\H_s'$-norm of $J'(v(s))$ in the last section. Hence combining \eqref{alt1.1} and \eqref{alt2} along with \eqref{eneq}, we obtain the following modified energy inequality with the $\vep$-regularization:
\begin{align}
\MoveEqLeft{
\frac 1 {q-1}
\left\| (\vep I + A^s)^{-1/2} \circ (-\Delta)^{-1} \circ J'(v(s)) \right\|_{H^1_0(\Omega)}^2
}\nonumber\\
&\leq \frac{4(q-1)}{q^2} \left\|
\partial_s (|v|^{(q-2)/2}v)(s)
\right\|_{L^2(\Omega)}^2\nonumber\\
&\stackrel{\eqref{eneq}}\leq - \frac{\d}{\d s} J(v(s))\label{eneqe2}
\end{align}
for a.e.~$s > 0$ and $\vep \in (0,1)$ (cf.~see \eqref{eneq2}).

\subsection{Quantitative gradient inequality with the $\vep$-regularization} 

We next derive a gradient inequality which better fit the present setting. 

\begin{lemma}[Quantitative gradient inequality with the $\vep$-regularization]\label{L:GIe}
There exist constants $s_1 \geq 0$ and $C > 0$ such that
\begin{align}
0 &\leq J(v(s)) - J(\phi) \nonumber \\
&\leq \left(
\frac{\vep \mu_k^s + 1}{2\nu_k^s} 
+ C \|v(s)-\phi\|_{H^1_0(\Omega)}^\rho
\right) \nonumber\\
&\quad \times \left\|
(\vep I + A^s)^{-1/2} \circ (-\Delta)^{-1} \circ J'(v(s))
\right\|_{H^1_0(\Omega)}^2\label{GIe}
\end{align}
for all $s \geq s_1$ and $\vep \in (0,1)$. Here $\nu^s_k$ denotes the smallest positive eigenvalue of $\Ls$ and $\mu^s_k = \nu^s_k + \lambda_q(q-1)$ {\rm (}see Remark {\rm \ref{R:Ls-inv}}{\rm )} and $\rho = \min\{1,q-2\} \in (0,1]$. Moreover, the constant $C$ depends only on $q$, $C_q$, $c(v)$ given in \eqref{vLq} and $\|\Lphi^{-1}\|_{\mathscr{L}(H^{-1}(\Omega),H^1_0(\Omega))}$ {\rm (}see Lemma {\rm\ref{L:Lsinv-bdd}}{\rm )}.
\end{lemma}

\begin{proof}
Since $\Ls = (-\Delta) \circ [I - \lambda_q(q-1)A^s]$ is invertible for $s > 0$ large enough (see (i) of Remark \ref{R:Ls-inv}), so is $I - \lambda_q (q-1) A^s$. Hence as in \eqref{alt0.5}, we find that
$$
\Ls^{-1} = \left[ I - \lambda_q(q-1)A^s \right]^{-1} \circ (-\Delta)^{-1},
$$
and therefore, we observe that
\begin{align*}
\lefteqn{
 \langle J'(v(s)), \Ls^{-1} \circ J'(v(s)) \rangle_{H^1_0(\Omega)}
}\\
&= \langle J'(v(s)), \left[ I - \lambda_q(q-1)A^s \right]^{-1} \circ (-\Delta)^{-1} \circ J'(v(s)) \rangle_{H^1_0(\Omega)}\\
&= \left( (-\Delta)^{-1} \circ J'(v(s)), \left[ I - \lambda_q(q-1)A^s \right]^{-1} \circ (-\Delta)^{-1} \circ J'(v(s)) \right)_{H^1_0(\Omega)}.
\end{align*}
Now, $\vep I + A^s$ is positive and self-adjoint in the Hilbert space $H^1_0(\Omega)$, and moreover, it is commutative with 
$$
I - \lambda_q(q-1)A^s = \left[ 1 + \vep \lambda_q(q-1) \right] I - \lambda_q(q-1) (\vep I + A^s).
$$
Therefore noting that
\begin{align*}
\MoveEqLeft{
\left( I - \lambda_q(q-1)A^s \right)^{-1} 
}\\
&= (\vep I + A^s)^{-1/2} \circ (\vep I + A^s) \circ \left[ I - \lambda_q(q-1)A^s \right]^{-1} \circ (\vep I + A^s)^{-1/2}, 
\end{align*}
we obtain
\begin{align}
\lefteqn{
 \langle J'(v(s)), \Ls^{-1} \circ J'(v(s)) \rangle_{H^1_0(\Omega)}
}\nonumber\\
&= \big( (\vep I + A^s)^{-1/2} \circ (-\Delta)^{-1} \circ J'(v(s)),\nonumber\\
&\quad (\vep I + A^s) \circ \left[ I - \lambda_q(q-1)A^s \right]^{-1} \circ (\vep I + A^s)^{-1/2} \circ (-\Delta)^{-1} \circ J'(v(s)) \big)_{H^1_0(\Omega)}\nonumber\\
&\leq \frac{\vep \mu_k^s + 1}{\mu_k^s - \lambda_q (q-1)} \left\|
(\vep I + A^s)^{-1/2} \circ (-\Delta)^{-1} \circ J'(v(s))
\right\|_{H^1_0(\Omega)}^2.
\label{alt3}
\end{align}
Here we have used the spectral decomposition of $A^s$ in $H^1_0(\Omega)$ in the last line, which indeed yields, for any $f\in H^1_0(\Omega)$, 
\begin{align*}
& \left( f,  (\vep I + A^s) \circ \left[ I - \lambda_q(q-1)A^s \right]^{-1} f\right)_{H^1_0(\Omega)}\\
& = \sum_{j=1}^\infty \frac{\vep + \lambda_j^s}{1 - \lambda_q (q-1) \lambda_j^s} \alpha_j^s (f)^2\\
& \leq \frac{\vep  + \lambda_k^s}{1 - \lambda_q (q-1)\lambda_k^s}  \sum_{j\geq k}  \alpha_j^s (f)^2 \leq \frac{\vep \mu_k^s + 1}{\mu_k^s - \lambda_q (q-1)} \| f\|_{H^1_0(\Omega)}^2. 
\end{align*}
Here we have set $\alpha_j^s (f) = (f, e^s_j)_{H^1_0(\Omega)}$ and also used the relation $\lambda_k^s=1/\mu_k^s$. Moreover, we have
\begin{align*}
\MoveEqLeft{
\|J'(v(s))\|_{H^{-1}(\Omega)}
}\nonumber\\
&=\left\|(\vep I + A^s)^{1/2} \circ (\vep I + A^s)^{-1/2} \circ (-\Delta)^{-1} \circ J'(v(s)) \right\|_{H^1_0(\Omega)}
\\
&\leq \|(\vep I + A^s)^{1/2}\|_{\mathscr{L}(H^1_0(\Omega))}
\left\|(\vep I + A^s)^{-1/2} \circ (-\Delta)^{-1} \circ J'(v(s)) \right\|_{H^1_0(\Omega)}\\
&\leq (\vep + \lambda^s_1)^{1/2} \left\|(\vep I + A^s)^{-1/2} \circ (-\Delta)^{-1} \circ J'(v(s)) \right\|_{H^1_0(\Omega)}.
\end{align*}
Here we also note that $\lambda^s_1$ is bounded for $s > 0$ (see Lemma \ref{L:ev-var} along with the boundedness of $v(s)$ in $L^q(\Omega)$ for $s \geq 0$). Consequently, recalling \eqref{GI-1} and \eqref{GIp}, we can derive \eqref{GIe} from \eqref{alt3}.
\end{proof}

Therefore combining \eqref{eneqe2} and \eqref{GIe}, we infer that
\begin{align*}
0 &\leq J(v(s)) - J(\phi)
\nonumber\\
&\stackrel{\eqref{GIe}}\leq
\left(
\frac{\vep \mu_k^s + 1}{2\nu_k^s} 
+ C \|v(s)-\phi\|_{H^1_0(\Omega)}^\rho
\right) \nonumber\\
&\quad \times
\left\|
(\vep I + A^s)^{-1/2} \circ (-\Delta)^{-1} \circ J'(v(s))
\right\|_{H^1_0(\Omega)}^2\nonumber\\
&\stackrel{\eqref{eneqe2}}\leq
- \left(
\frac{\vep \mu_k^s + 1}{2\nu_k^s}
+ C \|v(s)-\phi\|_{H^1_0(\Omega)}^\rho
\right) (q-1) \dfrac{\d}{\d s} J(v(s))
\end{align*}
for a.e.~$s > 0$ large enough. Hence passing to the limit as $\vep \to 0_+$, we obtain \eqref{conclu*} again for a.e.~$s > 0$ large enough. The rest of proof runs as before (see \S \ref{Ss:SRC}).

\section{Optimality of the convergence rate}\label{S:opt}

In this section, we shall prove Theorem \ref{T:opt}, which is concerned with the optimality of the rate of convergence \eqref{H10-ec} and \eqref{H10-stbl} obtained in Theorem \ref{T:sc-conv} and Corollary \ref{C:stbl} for nondegenerate \emph{least-energy} asymptotic profiles. To this end, we shall employ a novel ``linearization'' for the rescaled equation \eqref{eq:1.6} around an equilibrium  $\phi$ (cf.~see~\cite{BF21,McCann23}) as well as the results obtained so far. Moreover, it will also play a crucial role in the next section.

\begin{proof}[Proof of Theorem {\rm \ref{T:opt}}]
Let $\phi$ be a nondegenerate \emph{least-energy} solution to \eqref{eq:1.10}, \eqref{eq:1.11}, i.e., 
$$
J(\phi) = \inf_{w\in\mathcal S} J(w), 
$$
where $\mathcal S$ stands for the set of all nontrivial weak solutions to \eqref{eq:1.10}, \eqref{eq:1.11}. Then $\phi$ is always sign-definite in $\Omega$. Moreover, the least positive eigenvalue of \eqref{Lphi-ep} is the second one $\nu_2 = \mu_2 - \lambda_q (q-1) > 0$, that is, $k = 2$ (see~\cite[Lemma 1]{Lin}). Let $\xi_\vep \in H^1_0(\Omega)$, $\vep > 0$ satisfy \eqref{opt-hyp}. We set
\begin{equation}\label{phi_vep}
u_{0,\vep} := \phi + \xi_\vep = \phi + \mathbb{P}_2(\xi_\vep) + \mathbb{P}_2^\perp(\xi_\vep),
\end{equation}
where $\mathbb{P}_2$ denotes the spectral projection associated with \eqref{Lphi-ep} onto the eigenspace $E_2$ corresponding to $\nu_2 > 0$ and $\mathbb{P}_2^\perp := I - \mathbb{P}_2$. Set
\begin{equation*}
 v_{0,\vep} := c_\vep u_{0,\vep} \in \mathcal{X}, \quad c_\vep := t_*(u_{0,\vep})^{-1/(q-2)} > 0.
\end{equation*}
Then we note that
\begin{equation}\label{v0_vep}
v_{0,\vep} = \phi + (c_\vep - 1) \phi + c_\vep \mathbb{P}_2(\xi_\vep) + c_\vep \mathbb{P}_2^\perp(\xi_\vep)
\end{equation}
and $\phi$ is a principal eigenfunction of \eqref{Lphi-ep}; hence we have $\phi \in E_2^\perp$. Since $u_{0,\vep} \to \phi$ strongly in $H^1_0(\Omega)$ as $\vep \to 0_+$ and $t_* : H^1_0(\Omega) \to [0,\infty)$ is continuous (see~\cite[Proposition 4]{AK13}), it follows that $t_*(u_{0,\vep}) \to t_*(\phi) = 1$ (hence, $c_\vep \to 1$) as $\vep \to 0_+$. Thus $v_{0,\vep} \to \phi$ strongly in $H^1_0(\Omega)$ as $\vep \to 0_+$. 

Since $v_{0,\vep} \in \mathcal{X}$ is close (in $H^1_0(\Omega)$) enough to (nondegenerate) $\phi$ for $\vep > 0$ small enough, thanks to the exponential stability result (see Corollary \ref{C:stbl}), the energy solution $v_\vep = v_\vep(x,s)$ to the Cauchy-Dirichlet problem \eqref{eq:1.6}--\eqref{eq:1.8} with the initial datum $v_0 = v_{0,\vep}$ exponentially converges to $\phi$, that is,
\begin{equation}\label{vep-exp-c}
\left\|v_\vep(s) - \phi\right\|_{H^1_0(\Omega)}^2 \leq C \left( J(v_{0,\vep}) - J(\phi) \right)\e^{-\lambda_0 s} \quad \mbox{ for } \ s \geq 0.
\end{equation}
In particular, we also note that $\sup_{s \geq 0} \|v_\vep(s)\|_{H^1_0(\Omega)}$ is uniformly bounded for $\vep \in (0,1)$. 

Then we can find out the rate of the convergence $c_\vep \to 1$.

\begin{lemma}\label{L:cep}
It holds that $c_\vep = 1 + O(\vep)$ as $\vep \to 0_+$.
\end{lemma}

\begin{proof}
Thanks to Corollary 1 of~\cite{AK13}, we have the following estimate:
$$
\lambda_q \frac{\|u_{0,\vep}\|_{L^q(\Omega)}^q}{\|\nabla u_{0,\vep}\|_{L^2(\Omega)}^2} \leq t_*(u_{0,\vep}) \leq \lambda_q \frac{\|\phi\|_{L^q(\Omega)}^2}{\|\nabla \phi\|_{L^2(\Omega)}^2}\|u_{0,\vep}\|_{L^q(\Omega)}^{q-2},
$$
which gives
\begin{align*}
t_*(u_{0,\vep}) &= \lambda_q \frac{\|\phi\|_{L^q(\Omega)}^q}{\|\nabla \phi\|_{L^2(\Omega)}^2} + O(\vep) = 1 + O(\vep) \quad \mbox{ as } \ \vep \to 0_+.
\end{align*}
Moreover, we have $c_\vep = t_*(u_{0,\vep})^{-1/(q-2)} = 1 + O(\vep)$ as $\vep \to 0_+$. This completes the proof.
\end{proof}

By subtraction, one derives from \eqref{eq:1.6}--\eqref{eq:1.8} and \eqref{eq:1.10}, \eqref{eq:1.11} that
$$
\partial_s \left( (|v|^{q-2}v)(s) - |\phi|^{q-2}\phi \right) - \Delta (v(s)-\phi) = \lambda_q \left( (|v|^{q-2}v)(s) - |\phi|^{q-2}\phi \right)
$$
in $H^{-1}(\Omega)$ for a.e.~$s > 0$. Applying $(-\Delta)^{-1}$ to both sides and setting
$$
w(s) := (-\Delta)^{-1}\left( (|v|^{q-2}v)(s) - |\phi|^{q-2}\phi\right),
$$
we have
$$
\partial_s w(s) + v(s)-\phi = \lambda_q w(s) \ \mbox{ in } H^1_0(\Omega)\ \mbox{ for a.e. } s > 0.
$$
Set
$$
w_2(s) := \mathbb{P}_2 (w(s)),
$$
which solves
$$
\partial_s w_2(s) + \mathbb{P}_2(v(s)-\phi) = \lambda_q w_2(s) \ \mbox{ in } H^1_0(\Omega) \ \mbox{ for a.e. } s > 0.
$$
Define $A : H^1_0(\Omega) \to H^1_0(\Omega)$ by $A(z) = (-\Delta)^{-1}(|\phi|^{q-2} z)$ for $z \in H^1_0(\Omega)$ (as in \S \ref{Ss:As}) and note the relation,
\begin{equation}\label{proj-rel}
 \mathbb{P}_2 = \mu_2 \mathbb{P}_2 \circ A \ \mbox{ in } H^1_0(\Omega)\/{\rm ;}
\end{equation}
indeed, $\mathbb{P}_2$ is a spectral projection of $A$ corresponding to the eigenvalue $\lambda_2 = 1/\mu_2$. Then we find that
\begin{align}
\mathbb{P}_2(v(s)-\phi) &= \mu_2 \mathbb{P}_2 \circ A(v(s)-\phi)\nonumber\\
&= \mu_2 \mathbb{P}_2 \circ (-\Delta)^{-1}\left(|\phi|^{q-2}(v(s)-\phi)\right) \nonumber\\
&= \frac{\mu_2}{q-1} \mathbb{P}_2 \left(w(s) - (-\Delta)^{-1} \mathcal{R}(v(s),\phi) \right) \nonumber\\
&= \frac{\mu_2}{q-1} \left[ w_2(s) - \mathbb{P}_2 \left((-\Delta)^{-1} \mathcal{R}(v(s),\phi) \right)\right],\label{w2-trans}
\end{align}
where $\mathcal{R}(\cdot,\phi)$ is given by \eqref{R}. Thus we infer that
\begin{equation}\label{wk-eq}
\partial_s w_2(s) + \frac{\nu_2}{q-1} w_2(s) = \frac{\mu_2}{q-1} \mathbb{P}_2 \left((-\Delta)^{-1} \mathcal{R}(v(s),\phi)\right)
\end{equation}
in $H^1_0(\Omega)$ for a.e.~$s > 0$. Hence we obtain the formula,
\begin{align}
 w_2(s) &= \e^{-\frac{\nu_2}{q-1}s} w_2(0)\nonumber\\
&\quad + \frac{\mu_2}{q-1} \int^s_0 \e^{-\frac{\nu_2}{q-1}(s-\sigma)} \mathbb{P}_2 \left((-\Delta)^{-1} \mathcal{R}(v(\sigma),\phi)\right) \, \d \sigma\label{wk-inteq}
\end{align}
in $H^1_0(\Omega)$ for $s \geq 0$. Here we note from \eqref{vep-exp-c} that
\begin{align*}
 \|\mathbb{P}_2 \left((-\Delta)^{-1} \mathcal{R}(v(\sigma),\phi)\right)\|_{H^1_0(\Omega)}
&\leq \left\|(-\Delta)^{-1} \mathcal{R}(v(\sigma),\phi) \right\|_{H^1_0(\Omega)}\\
&= \left\|\mathcal{R}(v(\sigma),\phi) \right\|_{H^{-1}(\Omega)}\\
&\leq C_q \|\mathcal{R}(v(\sigma),\phi)\|_{L^{q'}(\Omega)}\\
&\stackrel{\eqref{R-est}}\leq C \|v(\sigma)-\phi\|_{L^q(\Omega)}^{\rho+1}\\
&\stackrel{\eqref{H10-stbl}}\leq C \left( J(v_{0,\vep})-J(\phi)\right)^{\frac{\rho+1}2} \e^{-\frac{\nu_2}{q-1}(\rho+1)\sigma},
\end{align*}
where $\rho := \min\{q-2,1\} \in (0,1]$. Moreover, we see that
\begin{align*}
w_2(0) &= \mathbb{P}_2 \left( (-\Delta)^{-1} \left(|v_{0,\vep}|^{q-2}v_{0,\vep} - |\phi|^{q-2}\phi \right)\right)\\
&= \mathbb{P}_2 \left( (-\Delta)^{-1} \left[ \mathcal{R}(v_{0,\vep},\phi) + (q-1) |\phi|^{q-2} (v_{0,\vep}-\phi) \right] \right)\\
&\stackrel{\eqref{proj-rel}}= \frac{q-1}{\mu_2} \mathbb{P}_2 (v_{0,\vep}-\phi) + \mathbb{P}_2 \left( (-\Delta)^{-1} \mathcal{R}(v_{0,\vep},\phi)\right),
\end{align*}
which along with \eqref{v0_vep} yields
\begin{align*}
\|w_2(0)\|_{H^1_0(\Omega)} 
&\geq \frac{q-1}{\mu_2} c_\vep \|\mathbb{P}_2(\xi_\vep)\|_{H^1_0(\Omega)}
- \left\| \mathcal{R}(v_{0,\vep},\phi)\right\|_{H^{-1}(\Omega)}\\
&\geq \frac{q-1}{\mu_2} c_\vep \|\mathbb{P}_2(\xi_\vep)\|_{H^1_0(\Omega)} - C\|v_{0,\vep}-\phi\|_{L^q(\Omega)}^{\rho+1}\\
&\stackrel{\eqref{opt-hyp}}> \frac{q-1}{2\mu_2} c_\vep \|\mathbb{P}_2(\xi_\vep)\|_{H^1_0(\Omega)}
\end{align*}
for $\vep > 0$ small enough. Here we used the fact (see \eqref{opt-hyp}) that
$$
\liminf_{\vep \to 0} c_\vep \vep^{-1} \|\mathbb{P}_2(\xi_\vep)\|_{H^1_0(\Omega)} > 0 \quad \mbox{ and }\quad \|v_{0,\vep}-\phi\|_{L^q(\Omega)}^{\rho+1} \leq C \vep^{\rho+1}.
$$
Therefore we infer that
\begin{align*}
\lefteqn{
\|w_2(s)\|_{H^1_0(\Omega)}
}\\
&\geq \e^{-\frac{\nu_2}{q-1}s} \|w_2(0)\|_{H^1_0(\Omega)} \\
&\quad - \frac{\mu_2}{q-1} \int^s_0 \e^{-\frac{\nu_2}{q-1}(s-\sigma)} \left\|\mathbb{P}_2 \left((-\Delta)^{-1} \mathcal{R}(v(\sigma),\phi)\right)\right\|_{H^1_0(\Omega)} \, \d \sigma \\
&\geq \frac{q-1}{2\mu_2} c_\vep \|\mathbb{P}_2(\xi_\vep)\|_{H^1_0(\Omega)} \e^{-\frac{\nu_2}{q-1}s}\\
&\quad - \frac{\mu_2}{q-1} C \left( J(v_{0,\vep})-J(\phi)\right)^{\frac{\rho+1}2} \int^s_0 \e^{-\frac{\nu_2}{q-1}(s-\sigma)} \e^{-\frac{\nu_2}{q-1}(\rho+1)\sigma} \, \d \sigma \\
&= \e^{-\frac{\nu_2}{q-1}s} \left[
\frac{q-1}{2\mu_2} c_\vep \|\mathbb{P}_2(\xi_\vep)\|_{H^1_0(\Omega)} + \frac{\mu_2}{\nu_2 \rho} C\left( J(v_{0,\vep})-J(\phi)\right)^{\frac{\rho+1}2} \left( \e^{-\frac{\nu_2\rho}{q-1} s}- 1\right) 
\right]\\
&\stackrel{\eqref{opt-hyp}}\geq \frac{q-1}{4\mu_2} \e^{-\frac{\nu_2}{q-1}s} c_\vep \|\mathbb{P}_2(\xi_\vep)\|_{H^1_0(\Omega)}
\end{align*}
for $\vep > 0$ small enough. Here we used the fact that
\begin{align*}
\MoveEqLeft
 J(v_{0,\vep})-J(\phi)\\
&= \frac 1 2 \|\nabla v_{0,\vep} - \nabla \phi\|_{L^2(\Omega)}^2 + \lambda_q \int_\Omega |\phi|^{q-2}\phi \left( v_{0,\vep}-\phi \right) \, \d x\nonumber\\
&\quad - \frac{\lambda_q}{q} \|v_{0,\vep}\|_{L^q(\Omega)}^q + \frac{\lambda_q}{q} \|\phi\|_{L^q(\Omega)}^q\nonumber\\
&\leq \frac 1 2 \|\nabla v_{0,\vep} - \nabla \phi\|_{L^2(\Omega)}^2 
+ o\left( \|v_{0,\vep}-\phi\|_{H^1_0(\Omega)}^2 \right)
\leq C \vep^2.
\end{align*}
The last inequality above follows from \eqref{v0_vep} and Lemma \ref{L:cep}. Thus recalling that
\begin{align*}
 \|w_2(s)\|_{H^1_0(\Omega)} &\leq \|w(s)\|_{H^1_0(\Omega)}\\
&= \left\| (|v|^{q-2}v)(s) - |\phi|^{q-2}\phi \right\|_{H^{-1}(\Omega)}\\
&\leq \|\mathcal{R}(v(s),\phi)\|_{H^{-1}(\Omega)} + (q-1)\||\phi|^{q-2}(v-\phi)\|_{H^{-1}(\Omega)}\\
&\leq C \e^{-\frac{\nu_2}{q-1}(\rho+1)s} + C\left( \int_\Omega |v(s)-\phi|^2 |\phi|^{q-2} \, \d x \right)^{1/2},
\end{align*}
we conclude that \eqref{opt-est} holds, that is, the rate of convergence \eqref{H10-ec} (and \eqref{H10-stbl}) turns out to be optimal. Thus Theorem \ref{T:opt} has been proved.
\end{proof}

\section{Faster decay for well-prepared data}\label{S:faster}

In this section, in contrast with the last section, we shall construct an energy solution $v = v(x,s)$ to \eqref{eq:1.6}--\eqref{eq:1.8} which converges to a nondegenerate \emph{least-energy} solution $\phi = \phi(x)$ to \eqref{eq:1.10}, \eqref{eq:1.11} at a rate \emph{faster than the optimal one} as $s \to +\infty$ (cf.~see Theorem \ref{T:sc-conv}). To be more precise, we shall find an initial datum $v_0 \in \mathcal X \setminus \{\phi\}$ close to $\phi$ such that
$$
\left\| v(s) - \phi \right\|_{H^1_0(\Omega)} \lesssim \e^{-\frac{\nu_m}{q-1}s} \quad \mbox{ for } \ s \geq 0,
$$
where $v$ denotes the energy solution to \eqref{eq:1.6}--\eqref{eq:1.8} with the initial datum $v_0$ and $\nu_m$ ($> \nu_2$) denotes the second positive eigenvalue of \eqref{Lphi-ep}.

Let $\vep > 0$ be a number, which will be fixed later. Let $\eta, \eta^\perp \in H^1_0(\Omega)$ satisfy
\begin{align}\label{5:ini}
\eta \in E_2, \quad \eta^\perp \in E_2^\perp \setminus \{0\}, \quad
\|\eta+\eta^\perp\|_{H^1_0(\Omega)} \leq \vep
\end{align}
and set
$$
u_0 := \phi + \eta + \eta^\perp.
$$
Let $v:=v(x,s)$ denote the energy solution to \eqref{eq:1.6}--\eqref{eq:1.8} with the initial datum $v_0 := c_{0} u_0 \in \mathcal{X}$, where $c_{0} := t_*(u_0)^{-1/(q-2)}$, that is,
$$
v_0 = \phi + (c_{0} - 1) \phi + c_{0} \eta + c_{0} \eta^\perp.
$$
From the stability result (see Corollary \ref{C:stbl}), there exists a constant $C \geq 0$ such that
\begin{equation}\label{fa:1}
\|v(s) - \phi\|_{H^1_0(\Omega)} \leq C \left( J(v_0)-J(\phi) \right) \e^{-\frac{\nu_2}{q-1}s} \quad \mbox{ for } \ s \geq 0,
\end{equation}
whenever $\vep > 0$ is small enough. Here we remark that $v_0 \neq \phi$ (hence $v(\cdot) \not\equiv \phi$), since $\eta^\perp \neq 0$. 

Recalling \eqref{wk-inteq}, we find that
\begin{align}
 w_2(s) &= \e^{-\frac{\nu_2}{q-1}s} w_2(0)\nonumber\\
&\quad + \frac{\mu_2}{q-1} \int^s_0 \e^{-\frac{\nu_2}{q-1}(s-\sigma)} \mathbb{P}_2 \left((-\Delta)^{-1} \mathcal{R}(v(\sigma),\phi)\right) \, \d \sigma\nonumber\\
&= \e^{-\frac{\nu_2}{q-1}s} \left[ w_2(0) + \frac{\mu_2}{q-1} \int^\infty_0 \e^{\frac{\nu_2}{q-1}\sigma} \mathbb{P}_2 \left((-\Delta)^{-1} \mathcal{R}(v(\sigma),\phi)\right) \, \d \sigma \right]\nonumber\\
&\quad - \frac{\mu_2}{q-1} \int^\infty_s \e^{-\frac{\nu_2}{q-1}(s-\sigma)} \mathbb{P}_2 \left((-\Delta)^{-1} \mathcal{R}(v(\sigma),\phi)\right) \, \d \sigma.
\label{w2-rel0}
\end{align}
Here we note from \eqref{fa:1} that
\begin{align*}
\lefteqn{
\left\| \int^\infty_s \e^{-\frac{\nu_2}{q-1}(s-\sigma)} \mathbb{P}_2 \left((-\Delta)^{-1} \mathcal{R}(v(\sigma),\phi)\right) \, \d \sigma \right\|_{H^1_0(\Omega)}
}\\
&\leq \int^\infty_s \e^{-\frac{\nu_2}{q-1}(s-\sigma)} \left\| \mathcal{R}(v(\sigma),\phi) \right\|_{H^{-1}(\Omega)} \, \d \sigma\\
&\stackrel{\eqref{R-est}}
\leq C\int^\infty_s \e^{-\frac{\nu_2}{q-1}(s-\sigma)} \left\| v(\sigma)-\phi \right\|_{L^q(\Omega)}^{\rho+1} \, \d \sigma\\
&\stackrel{\eqref{fa:1}}\lesssim 
\e^{-\frac{\nu_2}{q-1}(\rho+1)s} \quad \mbox{ for } \ s \geq 0,
\end{align*}
which decays faster than $\e^{-\frac{\nu_2}{q-1}s}$ as $s \to +\infty$. We shall find an initial datum $v_0 \in \mathcal X$ for which the energy solution $v = v(x,s)$ to \eqref{eq:1.6}--\eqref{eq:1.8} satisfies
\begin{align}\label{w2-rel}
w_2(0) + \frac{\mu_2}{q-1} \int^\infty_0 \e^{\frac{\nu_2}{q-1}\sigma} \mathbb{P}_2 \left((-\Delta)^{-1} \mathcal{R}(v(\sigma),\phi)\right) \, \d \sigma = 0.
\end{align}
Moreover, we observe that
\begin{align*}
 w_2(0) &= \mathbb{P}_2 \left( (-\Delta)^{-1} \left( |v_0|^{q-2}v_0 - |\phi|^{q-2}\phi \right) \right)\\
&= \mathbb{P}_2 \left( (-\Delta)^{-1} \left[ (q-1)|\phi|^{q-2}(v_0 - \phi) + \mathcal{R}(v_0,\phi)\right] \right)\\
&\stackrel{\eqref{proj-rel}}= \frac{q-1}{\mu_2} \mathbb{P}_2 \left( v_0-\phi \right) + \mathbb{P}_2 \left( (-\Delta)^{-1} \mathcal{R}(v_0,\phi) \right).
\end{align*}
Hence \eqref{w2-rel} is rewritten as
\begin{align}\label{w2-rel2}
\eta &= \eta - \mathbb{P}_2(v_0-\phi) - \frac{\mu_2}{q-1} \mathbb{P}_2 \left( (-\Delta)^{-1} \mathcal{R}(v_0,\phi) \right)\nonumber\\
&\quad - \left(\frac{\mu_2}{q-1}\right)^2 \int^\infty_0 \e^{\frac{\nu_2}{q-1}\sigma} \mathbb{P}_2 \left((-\Delta)^{-1} \mathcal{R}(v(\sigma),\phi)\right) \, \d \sigma \nonumber\\
&=: \Psi(\eta\,;\eta^\perp).
\end{align}
Here we recall that $v_0$ and $v$ are given as in the beginning of this section.
We first claim that $\Psi(\,\cdot\,;\eta^\perp)$ is a self-mapping on $\overline{B}_{\vep/2} := \{\eta \in E_2 \colon \|\eta\|_{H^1_0(\Omega)} \leq \vep/2\}$ for each $\eta^\perp \in E_2^\perp$ satisfying $\|\eta^\perp\|_{H^1_0(\Omega)} \leq \vep/2$ for $\vep > 0$ small enough. As in Lemma \ref{L:cep}, we see that
$$
c_{0} = t_*(\phi + \eta + \eta^\perp)^{-1/(q-2)} = 1 + O(\vep) \quad \mbox{ as } \ \vep \to 0_+
$$
\emph{uniformly} for $\eta \in E_2$ and $\eta^\perp \in E_2^\perp$ satisfying $\|\eta + \eta^\perp\|_{H^1_0(\Omega)} \leq \vep$. Thus one can take $\vep_0 > 0$ small enough that, for each $\vep \in (0,\vep_0)$,
\begin{align*}
\|\Psi(\eta;\eta^\perp)\|_{H^1_0(\Omega)}
&\leq |1 - c_{0}| \|\eta\|_{H^1_0(\Omega}
+ \frac{\mu_2}{q-1} \left\| \mathcal{R}(v_0,\phi) \right\|_{H^{-1}(\Omega)}\\
&\quad + \left(\frac{\mu_2}{q-1}\right)^2 \int^\infty_0 \e^{\frac{\nu_2}{q-1}\sigma} \left\|\mathcal{R}(v(\sigma),\phi)\right\|_{H^{-1}(\Omega)} \, \d \sigma\\
&\leq |1 - c_{0}| \|\eta\|_{H^1_0(\Omega} + O(\vep^{\rho+1})
\leq \frac{\vep}2
\end{align*}
for $\eta \in \overline{B}_{\vep/2}$ and $\eta^\perp \in E_2^\perp$ satisfying $\|\eta^\perp\|_{H^1_0(\Omega)} \leq \vep/2$. 
Now, we fix such an $\vep \in (0,\vep_0)$ (by taking account of \eqref{fa:1} as well) and an arbitrary $\eta^\perp \in E_2^\perp \setminus \{0\}$ satisfying $\|\eta^\perp\|_{H^1_0(\Omega)} \leq \vep/2$. We then claim that $\Psi(\,\cdot\,;\eta^\perp)$ is continuous in $\overline{B}_{\vep/2}$. Indeed, let $(\eta_n)$ be a sequence in $\overline{B}_{\vep/2}$ such that $\eta_n \to \eta$ for some $\eta \in \overline{B}_{\vep/2}$. Setting $u_{0,n} := \phi + \eta_n + \eta^\perp$ (and recalling $u_0 := \phi + \eta + \eta^\perp$), we first observe that $v_{0,n} := c_{0,n} u_{0,n} \to v_0 := c_{0} u_0$ strongly in $H^1_0(\Omega)$. Here $c_{0,n}$ and $c_{0}$ are defined for $u_{0,n}$ and $u_0$, respectively, and $c_{0,n} \to c_{0}$ from the continuity of $t_*(\cdot)$ in $H^1_0(\Omega)$ (see~\cite{AK13}). Hence it suffices to verify the continuity of the map $\Phi : \overline{B}_{\vep/2} \to \overline{B}_{\vep/2}$ given by
$$
\Phi(\eta) := \int^\infty_0 \e^{\frac{\nu_2}{q-1}\sigma} \mathbb{P}_2 \left((-\Delta)^{-1} \mathcal{R}(v(\sigma),\phi)\right) \, \d \sigma \quad \mbox{ for } \ \eta \in \overline{B}_{\vep/2}.
$$
Actually, we observe that
\begin{align*}
\lefteqn{
 \|\Phi(\eta_n) - \Phi(\eta)\|_{H^1_0(\Omega)}
}\\
&\leq \int^\infty_0 \e^{\frac{\nu_2}{q-1}\sigma} \left\| \mathcal{R}(v(\sigma),\phi) - \mathcal{R}(v_n(\sigma),\phi) \right\|_{H^{-1}(\Omega)}\, \d \sigma\\
&= \int^S_0 \e^{\frac{\nu_2}{q-1}\sigma} \left\| \mathcal{R}(v(\sigma),\phi) - \mathcal{R}(v_n(\sigma),\phi) \right\|_{H^{^1}(\Omega)}\, \d \sigma\\
&\quad + \int^\infty_S \e^{\frac{\nu_2}{q-1}\sigma} \left\| \mathcal{R}(v(\sigma),\phi) - \mathcal{R}(v_n(\sigma),\phi) \right\|_{H^{-1}(\Omega)}\, \d \sigma,
\end{align*}
where $v_n = v_n(x,s)$ denotes the energy solution to \eqref{eq:1.6}--\eqref{eq:1.8} for the initial datum $v_{0,n}$, for $S > 0$. For any $\nu>0$, one can take $S_\nu > 0$ large enough that
\begin{align*}
\MoveEqLeft{
\int^\infty_{S_\nu} \e^{\frac{\nu_2}{q-1}\sigma} \left\| \mathcal{R}(v(\sigma),\phi) - \mathcal{R}(v_n(\sigma),\phi) \right\|_{H^{-1}(\Omega)}\, \d \sigma
}\\
&\lesssim \int^\infty_{S_\nu} \e^{\frac{\nu_2}{q-1}\sigma} \e^{-\frac{\nu_2}{q-1}(\rho+1)\sigma} \, \d \sigma
< \frac{\nu}2.
\end{align*}
Moreover, due to the continuous dependence of energy solutions to \eqref{eq:1.6}--\eqref{eq:1.8} on initial data, we can take $N_\nu \in \N$ such that
\begin{align*}
\lefteqn{
\int^{S_\nu}_0 \e^{\frac{\nu_2}{q-1}\sigma} \left\| \mathcal{R}(v(\sigma),\phi) - \mathcal{R}(v_n(\sigma),\phi) \right\|_{H^{-1}(\Omega)}\, \d \sigma
}\\
&\stackrel{\eqref{R}}\leq \int^{S_\nu}_0 \e^{\frac{\nu_2}{q-1}\sigma} \left\| (|v|^{q-2}v)(\sigma) - (|v_n|^{q-2}v_n)(\sigma) \right\|_{H^{-1}(\Omega)}\, \d \sigma\\
&\quad + (q-1)C_q \int^{S_\nu}_0 \e^{\frac{\nu_2}{q-1}\sigma} \|\phi\|_{L^q(\Omega)}^{q-2} \left\| v(\sigma) - v_n(\sigma) \right\|_{L^q(\Omega)}\, \d \sigma
< \frac\nu 2
\end{align*}
for $n \geq N_\nu$. Thus $\Phi$ is continuous in $\overline{B}_{\vep/2}$, and so is $\Psi(\,\cdot\,;\eta^\perp)$. Combining all these facts and employing Brower's fixed point theorem, we conclude that there exists $\eta_* \in \overline{B}_{\vep/2}$ such that $\Psi(\eta_*\,;\eta^\perp) = \eta_*$. Thus we have proved that
\begin{equation}\label{w2-decay}
\|w_2(s)\|_{H^1_0(\Omega)} \lesssim \e^{-\frac{\nu_2}{q-1}(\rho+1)s}
\end{equation}
for such well-prepared initial data $u_0 = \phi + \eta_* + \eta^\perp$ (i.e., $v_0 = t_*(u_0)^{-1/(q-2)}u_0$). From \eqref{w2-trans} and \eqref{w2-decay}, we note that
\begin{align}
 \|\mathbb{P}_2(v(s)-\phi)\|_{H^1_0(\Omega)}
 &\leq \frac{\mu_2}{q-1} \|w_2(s)\|_{H^1_0(\Omega)}
+ \frac{\mu_2}{q-1} \left\| \mathcal{R}(v(s),\phi)\right\|_{H^{-1}(\Omega)}\nonumber\\
 &\lesssim \e^{-\frac{\nu_2}{q-1}(\rho+1)s}
\quad \mbox{ for } \ s \geq 0.\label{wj-decay}
\end{align}

Now, let us go back to the proof of Theorem \ref{T:sc-conv}. In particular, we recall the key identity (see \eqref{fLinvf_0} and \eqref{fLinvf}),
$$
\langle \bar f, \Ls^{-1} \bar f \rangle_{H^1_0(\Omega)}
= \sum_{j=1}^\infty (\beta^s_j)^2 \frac{\mu_j^s}{\mu_j^s - \lambda_q (q-1)},
\quad \beta^s_j = \langle \bar f, e^s_j \rangle_{H^1_0(\Omega)},
$$
where $\bar f$ is the zero extension of $f$ onto $\Omega$, for $f \in \H_s'$. Let $\nu_m = \mu_m - \lambda_q (q-1) > \nu_2$ be the second positive eigenvalue of \eqref{Lphi-ep}, that is, $\nu_2 = \cdots = \nu_{m-1} < \nu_m$ (hence $m > 2$). Here we derive instead of \eqref{fLinvf} that 
\begin{align}
\lefteqn{
\langle \bar f, \Ls^{-1} \bar f \rangle_{H^1_0(\Omega)}
}\nonumber\\
&\leq \sum_{j=2}^\infty (\beta^s_j)^2 \frac{\mu_j^s}{\mu_j^s - \lambda_q (q-1)}\nonumber\\
&= \sum_{j=m}^\infty (\beta^s_j)^2 \frac{\mu_j^s}{\mu_j^s - \lambda_q (q-1)} + \sum_{j=2}^{m-1} (\beta^s_j)^2 \frac{\mu_2^s}{\mu_2^s - \lambda_q (q-1)} \nonumber\\
&\leq \frac{1}{\mu_m^s - \lambda_q (q-1)} \sum_{j=m}^\infty (\beta^s_j)^2 \mu_j^s + \sum_{j=2}^{m-1} (\beta^s_j)^2 \frac{\mu_2^s}{\mu_2^s - \lambda_q (q-1)} \nonumber\\
&= \frac{1}{\nu_m^s} \sum_{j=2}^\infty (\beta^s_j)^2 \mu_j^s - \frac{1}{\mu_m^s - \lambda_q (q-1)} \sum_{j=2}^{m-1} (\beta^s_j)^2 \mu_2^s \nonumber \\
&\quad + \frac{1}{\mu_2^s - \lambda_q (q-1)} \sum_{j=2}^{m-1} (\beta^s_j)^2 \mu_2^s\nonumber\\
&\leq \frac{1}{\nu_m^s} \sum_{j=2}^\infty \|f\|_{\H_s'}^2 + r_2(s),\label{fLinvf_mod}
\end{align}
where $r_2(s)$ is given by
$$
r_2(s) := \left( - \frac{1}{\mu_m^s - \lambda_q (q-1)} + \frac{1}{\mu_2^s - \lambda_q (q-1)} \right) \sum_{j=2}^{m-1} (\beta^s_j)^2 \mu_2^s.
$$
Then we observe that
$$
|r_2(s)| \leq C \sum_{j=2}^{m-1} (\beta^s_j)^2 = C \|(\mathbb{P}^s_2)^*(\bar f)\|_{H^{-1}(\Omega)}^2,
$$
where $(\mathbb{P}_2^s)^* : H^{-1}(\Omega) \to H^{-1}(\Omega)$ stands for the adjoint operator of the spectral projection $\mathbb{P}_2^s : H^1_0(\Omega) \to H^1_0(\Omega)$ of $A^s$ corresponding to the least eigenvalue $\mu_2^s > \lambda_q(q-1)$ of \eqref{ep}, that is,
$$
(\mathbb{P}^s_2)^*(f) = \sum_{j=2}^{m-1} \langle f, e^s_j \rangle_{H^1_0(\Omega)}( -\Delta e^s_j) \quad \mbox{ for } \ f \in H^{-1}(\Omega).
$$
We also note that $(\mathbb{P}^s_2)^* \circ (-\Delta) = (-\Delta) \circ \mathbb{P}^s_2$. \prf{Indeed, by definition, we have $\langle (\mathbb{P}^s_2)^* f, u \rangle_{H^1_0(\Omega)} = \langle f, \mathbb{P}^s_2 u \rangle_{H^1_0(\Omega)}$ for $u \in H^1_0(\Omega)$ and $f \in H^{-1}(\Omega)$. Substitute $f = -\Delta v$ for $v \in H^1_0(\Omega)$. We observe that $\langle (\mathbb{P}^s_2)^* \circ (-\Delta) v, u \rangle_{H^1_0(\Omega)} = \langle -\Delta v, \mathbb{P}^s_2 u \rangle_{H^1_0(\Omega)} = (v,\mathbb{P}^s_2 u)_{H^1_0(\Omega)} = (\mathbb{P}^s_2 v, u)_{H^1_0(\Omega)} = \langle (-\Delta) \circ \mathbb{P}^s_2 v, u\rangle_{H^1_0(\Omega)}$. Thus the relation follows from the arbitrariness of $u,v \in H^1_0(\Omega)$.} Substitute $f = J'(v(s)) \in \H_s'$ and note that
\begin{align*}
 \|J'(v(s)) - \Ls (v(s)-\phi)\|_{L^{q'}(\Omega)}
\leq C \|v(s)-\phi\|_{L^q(\Omega)}^{\rho+1} \quad \mbox{ for } \ s \geq 0
\end{align*}
(as in the proof of Lemma \eqref{L:Taylor}). Since $(\mathbb{P}^s_2)^* \circ \Ls = (\mathbb{P}^s_2)^* \circ (-\Delta) \circ [I - \lambda_q(q-1) A^s] = (-\Delta) \circ \mathbb{P}^s_2 \circ [I - \lambda_q (q-1)A^s] = \Ls \circ \mathbb{P}^s_2$, it then follows from \eqref{H10-ec} and \eqref{wj-decay} that
\begin{align*}
|r_2(s)| 
&\leq C \|(\mathbb{P}^s_2)^*(J'(v(s))\|_{H^{-1}(\Omega)}^2\\
&\leq C \|(\mathbb{P}^s_2)^* \circ \Ls (v(s)-\phi)\|_{H^{-1}(\Omega)}^2 + C \|v(s)-\phi\|_{L^q(\Omega)}^{\rho+1}\\
&= C \|\Ls \circ \mathbb{P}^s_2 (v(s)-\phi)\|_{H^{-1}(\Omega)}^2 + C \|v(s)-\phi\|_{L^q(\Omega)}^{\rho+1}\\
&\leq C \|\Ls\|_{\mathscr{L}(H^1_0(\Omega),H^{-1}(\Omega))}^2 \|\mathbb{P}^s_2 (v(s)-\phi)\|_{H^1_0(\Omega)}^2 + C \|v(s)-\phi\|_{L^q(\Omega)}^{\rho+1}\\
&\leq C \|\Ls\|_{\mathscr{L}(H^1_0(\Omega),H^{-1}(\Omega))}^2 \|\mathbb{P}_2 (v(s)-\phi)\|_{H^1_0(\Omega)}^2 + C \|v(s)-\phi\|_{L^q(\Omega)}^{\rho+1}\\
&\lesssim \e^{-\frac{2\nu_2}{q-1}(\rho+1)s}.
\end{align*}
Hence as in \eqref{GI} we can obtain
\begin{align}
\lefteqn{
 J(v(s)) - J(\phi)
}\nonumber\\
&\leq \left( \frac 1 {2 \nu_m^s} + C \|v(s)-\phi\|_{H^1_0(\Omega)}^\rho \right) \| J'(v(s)) \|_{\H_s'}^2 + C \e^{-\frac{2\nu_2}{q-1}(\rho+1)s}
\label{GI2}
\end{align}
for all $s \geq s_1$ large enough and some constant $C \geq 0$ independent of $s$. Thus one can verify that
\begin{align}
\dfrac{\d H}{\d s}(s) + \frac{2\nu_m}{q-1} H(s)
&\leq C \|v(s)-\phi\|_{H^1_0(\Omega)}^\rho H(s) + C \e^{-\frac{2\nu_2}{q-1}(\rho+1)s}\nonumber\\
&\leq C \e^{-\frac{2\nu_2}{q-1}\left(\frac\rho2+1\right) s},
\label{H-ineq-2}
\end{align}
where $H(s) := J(v(s))-J(\phi)$, for $s \geq s_1$. In case $\nu_2\left(\rho/2+1\right) > \nu_m$, we immediately obtain
$$
H(s) \leq C \e^{-\frac{2\nu_m}{q-1}s},
$$
which together with Lemma \ref{L:coer} implies the desired conclusion. In case $\nu_2\left(\rho/2+1\right) < \nu_m$, it follows that
$$
H(s) \leq C \left( \e^{-\frac{2\nu_m}{q-1}s} + \e^{-\frac{2\nu_2}{q-1}\left(\frac\rho2+1\right)s} \right) \quad \mbox{ for } \ s \geq 0.
$$
Due to Lemma \ref{L:coer}, we obtain
$$
\|v(s)-\phi\|_{H^1_0(\Omega)} \leq C \left( \e^{-\frac{\nu_m}{q-1}s} + \e^{-\frac{\nu_2}{q-1}\left(\frac\rho2+1\right) s} \right) \quad \mbox{ for } \ s \geq 0.
$$
Using the above improved decay estimate instead of the original one (see \eqref{H10-ec}) and repeating the argument so far, we can improve the estimates for $w_2$ and then obtain
$$
\|v(s)-\phi\|_{H^1_0(\Omega)} \leq C \left( \e^{-\frac{\nu_m}{q-1}s} + \e^{-\frac{\nu_2}{q-1}\left(\frac\rho2+1\right)^2s} \right) \quad \mbox{ for } \ s \geq 0.
$$
Hence repeating this procedure (in finite time) and noting that $\rho > 0$, we conclude that
\begin{equation}\label{conclu}
\|v(s)-\phi\|_{H^1_0(\Omega)} \leq C \e^{-\frac{\nu_m}{q-1}s} \quad \mbox{ for } \ s \geq 0.
\end{equation} 
In case $\nu_2\left(\rho/2+1\right) = \nu_m$, we can derive that
$$
H(s) \leq C \e^{-\frac{2\nu}{q-1}s}
$$
for any $\nu_2 < \nu < \nu_m$. Hence we can also obtain \eqref{conclu} as in the last case. Thus we have obtained

\begin{theorem}[Faster decay for well-prepared initial data]\label{T:faster}
Let $\Omega$ be any bounded domain of $\R^N$ with boundary $\partial \Omega$. Assume \eqref{hypo} and let $\phi$ be a nondegenerate \emph{least-energy} solution to \eqref{eq:1.10}, \eqref{eq:1.11}. Let $\nu_m > \nu_2$ be the second positive eigenvalue of \eqref{Lphi-ep}. Then there exists $v_0 \in \mathcal{X} \setminus \{\phi\}$ such that
$$
\|v(s)-\phi\|_{H^1_0(\Omega)} \leq C \e^{-\frac{\nu_m}{q-1}s} \quad \mbox{ for } \ s \geq 0,
$$
where $v = v(x,s)$ is the energy solution to \eqref{eq:1.6}--\eqref{eq:1.8} with the initial datum $v_0$.
\end{theorem}

\appendix

\section{Proof of Lemma \ref{L:ev-var}}\label{A:ev-var}

In this appendix, we shall give a proof of Lemma \ref{L:ev-var} for the convenience of the reader. Due to \eqref{u-ext}, for each $w \in H^1_0(\Omega)$, there exists $w^s_0 \in E^s_0$ such that
$$
w = w^s_0 + \sum_{j=1}^\infty \alpha^s_j e^s_j \ \mbox{ in } H^1_0(\Omega), \quad \alpha^s_j := (w, e^s_j)_{H^1_0(\Omega)}.
$$
Hence, when $\|w\|_{H^1_0(\Omega)} = 1$, we have
$$
\|w^s_0\|_{H^1_0(\Omega)}^2 + \sum_{j=1}^\infty (\alpha^s_j)^2 = 1.
$$
Then we observe that
\begin{align*}
\int_\Omega |v(s)|^{q-2} w^2 \, \d x
&= \sum_{i=1}^\infty \sum_{j=1}^\infty \alpha^s_i \alpha^s_j \int_\Omega |v(s)|^{q-2} e^s_i e^s_j \, \d x.
\end{align*}
Here we used the fact that $|v(s)|^{q-2}w^s_0 = 0$ a.e.~in $\Omega$. Hence we see that
\begin{align}
\int_\Omega |v(s)|^{q-2} w^2 \, \d x
&= \sum_{i=1}^\infty \sum_{j=1}^\infty \alpha^s_i \alpha^s_j \frac 1 {\mu^s_i} \int_\Omega \nabla e^s_i \cdot \nabla e^s_j \, \d x \nonumber\\
& = \sum_{j=1}^\infty \frac{(\alpha^s_j)^2}{\mu^s_j}.\label{Hs-norm-decomp}
\end{align}

Set $Y = \mathrm{span} \{ e^s_1, e^s_2, \ldots, e^s_j\}$. Then we find from \eqref{Hs-norm-decomp} that
$$
\inf_{\substack{w \in Y\\[0mm]\|w\|_{H^1_0(\Omega)}=1}} \int_\Omega |v(s)|^{q-2} w^2 \, \d x  = \frac 1 {\mu^s_j},
$$
which implies
\begin{equation}\label{1}
\sup_{\substack{Y \subset H^1_0(\Omega)\\[0mm]\dim Y = j}} \inf_{\substack{w \in Y\\[0mm]\|w\|_{H^1_0(\Omega)}=1}} \int_\Omega |v(s)|^{q-2} w^2 \, \d x \geq \frac 1 {\mu^s_j}.
\end{equation}
We shall prove the inverse inequality. In case $j = 1$, thanks to \eqref{Hs-norm-decomp}, we note that $\mu^s_1$ can be characterized as follows:
\begin{align*}
\MoveEqLeft
\sup_{\|w\|_{H^1_0(\Omega)} = 1} \int_\Omega |v(s)|^{q-2} w^2 \, \d x 
\\
&= \sup \left\{ \sum_{j=1}^\infty \frac{(\alpha^s_j)^2}{\mu^s_j} \colon \sum_{j=1}^\infty (\alpha^s_j)^2 \leq 1 \right\} = \frac 1 {\mu^s_1},
\end{align*}
which in particular implies the inverse inequality of \eqref{1} with $j = 1$. In case $j \geq 2$, for each $j$-dimensional subspace $Y$ of $H^1_0(\Omega)$, one can take $w_Y \in Y$ such that $\|w_Y\|_{H^1_0(\Omega)} = 1$ and $(w_Y, e^s_i)_{H^1_0(\Omega)} = 0$ for $i = 1,\ldots,j-1$. Hence it holds that
$$
\int_\Omega |v(s)|^{q-1} w_Y^2 \, \d x \leq \frac 1 {\mu^s_j};
$$
whence it follows that
$$
\inf_{\substack{w \in Y\\[0mm]\|w\|_{H^1_0(\Omega)}=1}} \int_\Omega |v(s)|^{q-2} w^2 \, \d x \leq \frac 1 {\mu^s_j}
$$
for $j \geq 2$. Thus the inverse inequality of \eqref{1} follows from the arbitrariness of $Y$. Therefore we obtain \eqref{ev-var-s}. The assertion \eqref{ev-var} for $\mu_j$ can be verified in the same fashion. \qed


\end{document}